\documentclass[a4paper]{amsproc}
\usepackage{amssymb}
\usepackage{mathrsfs}
\usepackage{amsfonts}
\usepackage{amsthm}
\usepackage{amsmath,amscd}
\usepackage{amsxtra}     
\usepackage{bm}
\usepackage[all,cmtip]{xy}
\usepackage{epsfig}
\usepackage{verbatim}
\usepackage{color}
\usepackage{enumerate}
\usepackage{hyperref}
\usepackage{tikz}
\usetikzlibrary{arrows,matrix,shapes,trees}

\theoremstyle{plain}
 \newtheorem{thm}{Theorem}[section]
 \newtheorem*{theorem*}{Theorem}
 \newtheorem{prop}[thm]{Proposition}
 \newtheorem{lem}[thm]{Lemma}
 \newtheorem{cor}[thm]{Corollary}
 \newtheorem*{claim}{Claim}
\theoremstyle{definition}
 
 \newtheorem{defn}[thm]{Definition}
\theoremstyle{remark}
 \newtheorem{rmk}[thm]{Remark}

\newcommand{\N}{{\mathbb N}}
\newcommand{\Q}{{\mathbb Q}}
\newcommand{\Z}{{\mathbb Z}}

\newcommand{\Gm}{\mathbb{G}_{\mr{m}}}

\newcommand{\Gml}{\mathbb{G}_{\mr{m,log}}}

\newcommand{\mr}{\mathrm}
\newcommand{\mc}{\mathcal}

\newcommand{\fsS}{(\mr{fs}/S)}
\newcommand{\fsSet}{(\mr{fs}/S)_{\mr{\acute{e}t}}}
\newcommand{\fsSfl}{(\mr{fs}/S)_{\mr{fl}}}
\newcommand{\fsSket}{(\mr{fs}/S)_{\mr{k\acute{e}t}}}
\newcommand{\fsSkfl}{(\mr{fs}/S)_{\mr{kfl}}}
\newcommand{\Set}{S_{\mr{\acute{e}t}}}
\newcommand{\Sket}{S_{\mr{k\acute{e}t}}}
\newcommand{\Sfl}{S_{\mr{fl}}}
\newcommand{\Skfl}{S_{\mr{kfl}}}
\newcommand{\Spec}{\mathop{\mr{Spec}}}
\newcommand{\gp}{\mathrm{gp}}
\newcommand{\et}{\mathrm{\acute{e}t}}

\title[Log $p$-divisible groups associated to log 1-motives]{Log $p$-divisible groups associated to log 1-motives}



\author[Matti W\"urthen]{\bfseries Matti W\"urthen}
\author[Heer Zhao]{\bfseries Heer Zhao}

\address{
    Matti W\"urthen, 
    Institut f\"ur Mathematik , 
    Goethe-Universit\"at Frankfurt, 
    Frankfurt am Main 60325 
    Germany, 
    	wuerthen@math.uni-frankfurt.de}
    	
\address{
    Heer Zhao, 
    Fakult\"at f\"ur Mathematik, 
    Universit\"at Duisburg-Essen, 
    Essen 45117, 
    Germany, 
    	heer.zhao@uni-due.de}

\begin{document}

\vspace{18mm} \setcounter{page}{1} \thispagestyle{empty}

\begin{abstract}
We first provide a detailed proof of Kato's classification theorem of log $p$-divisible groups over a noetherian henselian local ring. Exploring Kato's idea further, we then define the notion of a standard extension of a classical finite \'etale group scheme (resp. classical \'etale $p$-divisible group) by a classical finite flat group scheme (resp. classical $p$-divisible group) in the category of finite Kummer flat group log schemes (resp. log $p$-divisible groups), with respect to a given chart on the base. These results are then used to prove that log $p$-divisible groups are formally log smooth. We then study the finite Kummer flat group log schemes $T_n(\mathbf{M}):=H^{-1}(\mathbf{M}\otimes_{\Z}^L\Z/n\Z)$ (resp. the log $p$-divisible group $\mathbf{M}[p^{\infty}]$) of a log 1-motive $\mathbf{M}$ over an fs log scheme and show that they are \'etale locally standard extensions. Lastly, we give a proof of the Serre-Tate theorem for log abelian varieties with constant degeneration.
\end{abstract}

\maketitle

\section*{Introduction}
When studying degenerations of abelian varieties one is led to Kato's theory of finite Kummer flat group log schemes and log $p$-divisible groups (see \cite{kat4} and also \cite[App.]{zha3}\footnote{The terminology ``log finite flat group scheme'' in \cite[App.]{zha3} is replaced by ``finite Kummer flat group log schemes'' here which should be more suitable.} for a brief account of finite Kummer flat group log schemes). For instance, by a theorem of Kato (see \cite[Theorem 1.3]{zha5}), the $p$-divisible group of a semistable abelian variety over a complete discrete valuation field extends to a log $p$-divisible group over the corresponding discrete valuation ring. In this paper we will expand on some ideas developed in \cite{kat4} and connect them with the theory of log $1$-motives and log abelian varieties.

Let $S$ be an fs log scheme endowed with a suitable global chart, whose underlying scheme is the spectrum of a noetherian henselian local ring with positive residue characterisitic $p$. When studying finite Kummer flat group log schemes over $S$, one key point is to understand extensions of classical finite \'etale group schemes by  classical finite flat group schemes over $S$ in a logarithmic category, see \cite[the pragraph before Thm. A.3]{zha3} for an explanation of this. In \cite[Thm. 3.3]{kat4}, see also Theorem \ref{thm2.1}, Kato gives a description of such an extension in terms of a classical extension together with a certain monodromy datum. We call such extensions standard (see Definition \ref{defn2.1}).\\

As a corollary of Kato's theorem (Theorem \ref{thm2.1}), one then also gets a description of the extensions of a classical \'etale $p$-divisible group by a classical $p$-divisible group over $S$, see Theorem \ref{thm2.3}. After presenting the basic objects in the first section, we then in the second section present a complete proof to Kato's theorem following Kato's approach. In the procedure we explore Kato's idea further and define the notion of a standard extension of a classical finite \'etale group scheme (resp. classical \'etale $p$-divisible group) by a classical finite flat group scheme (resp. classical $p$-divisible group), see Definition \ref{defn2.1} (resp. Definition \ref{defn2.2}). Then Kato's results amount to saying that, over a noetherian henselian local fs log scheme admitting a global chart, any extension of a classical finite \'etale group scheme (resp. classical \'etale $p$-divisible group) by a classical finite flat group scheme (resp. classical $p$-divisible group) is always standard. Over a general base, we
get the following:
\begin{theorem*}(See Theorem \ref{thm2.2})
Let $S$ be a locally noetherian fs log scheme and let $G$ be a finite Kummer flat group log scheme over $S$, which is an extension of a classical finite \'etale group scheme by a classical finite flat group scheme. Then $G$ is a standard extension \'etale locally on $S$.
\end{theorem*}
Moreover, as an application of Kato's results, we can prove the following theorem for objects in the category $(\text{$p$-div}/S)^{\mr{log}}_{\mr{d}}$, which consists of log $p$-divisible groups whose dual is also a log $p$-divisible group. This is analogous to the formal smoothness of classical $p$-divisible groups (see \cite[Chapter II 3.3.13]{mes1}).

\begin{theorem*}(See Theorem \ref{thm2.4})
Let $S$ be a locally noetherian fs log scheme on which $p$ is locally nilpotent, and $H$ a log $p$-divisible group over $S$, which lies in $(\text{$p$-div}/S)^{\mr{log}}_{\mr{d}}$. Then $H$ is formally log smooth, i.e. for any strict closed square-zero thickening $T_0\hookrightarrow T$ in $(\mr{fs}/S)$, any element of $H(T_0)$ can be lifted to an element of $H(T)$ \'etale locally on $T$.
\end{theorem*}

In the third section, we study the finite Kummer flat group log scheme $$T_n(\mathbf{M}):=H^{-1}(\mathbf{M}\otimes_{\Z}^L\Z/n\Z)$$
for a log 1-motive $\mathbf{M}=[Y\xrightarrow{u}G_{\mr{log}}]$ over a locally noetherian fs log scheme $S$, as well as the log $p$-divisible group $\mathbf{M}[p^{\infty}]:=\varinjlim_n T_{p^n}(\mathbf{M})$ of $\mathbf{M}$. Let $T$ be the torus part of $G$ and let $X$ be the character group of $T$. The composition
\[Y\xrightarrow{u}G_{\mr{log}}\to G_{\mr{log}}/G\cong T_{\mr{log}}/T=\mc{H}om(X,\Gml/\Gm)\]
gives rise to a pairing $\langle-,-\rangle:X\times Y\to \Gml/\Gm$, see \cite[\S 2.3]{k-k-n2}. We call this pairing the \textbf{monodromy pairing of $\mathbf{M}$}. In Proposition \ref{prop3.4}, we show that \'etale locally the monodromy pairing of $\mathbf{M}$ gives rise to a canonical pseudo-monodromy of $T_n(\mathbf{M})$, as well as of $\mathbf{M}[p^\infty]$.

In the last section we turn to the Serre-Tate theorem for log abelian varieties with constant degeneration:
\begin{theorem*}(see Theorem \ref{thm5.1})
Let $S_{0}$ be a locally noetherian fs log scheme on which $p$ is locally nilpotent and let $S_{0}\subset S$ be a strict infinitesimal thickening of fs log schemes. Let $A_{0}$ be a log abelian variety with constant degeneration over $S_{0}$. Let $\mr{Def}_{A_{0}}(S)$ be the category of pairs $(A, \phi)$ where $A$ is log abelian variety over $S$, and $\phi:A_{S_{0}}\xrightarrow{\cong}A_{0}$ is an isomorphism. Similarly let $\mr{Def}_{A_{0}[p^{\infty}]}(S)$ be the category of pairs $(H, \psi)$, where $H$ is a log $p$-divisible group over $S$, with an identification $\psi:H_{S_{0}}\xrightarrow{\cong} A_{0}[p^{\infty}]$. \\
Then the functor $A\mapsto A[p^{\infty}]$, taking a log abelian variety over $S$ to its associated log $p$-divisible group, induces an equivalence of categories 
\begin{center}
$\mr{Def}_{A_{0}}(S)\leftrightarrow \mr{Def}_{A_{0}[p^{\infty}]}(S)$.
\end{center}
\end{theorem*}
For this we follow Drinfeld's approach for abelian varieties as in \cite[\S 1]{katz1}. The key point here is to verify the hypothesis about formal (log) smoothness from \cite[Lem. 1.1.3]{katz1} for log $p$-divisible groups in $(\text{$p$-div}/S)^{\mr{log}}_{\mr{d}}$ and log abelian varieties with constant degeneration.

The above theorem can also be regarded as a Serre-Tate theorem for pointwise polarizable log 1-motives. There is also a Serre-Tate theorem for classical $1$-motives over local Artin rings, which is the main theorem of \cite{b-m1}.

\section{Finite Kummer flat group log schemes and log $p$-divisible groups}

\subsection{Kummer log topologies}
Unless otherwise stated, we always denote by $S$ a locally noetherian fs log scheme. By log structure we always mean log structure on the classical \'etale site. We denote by $(\mr{fs}/S)$ the category of fs log schemes over $S$. We recall the Kummer log flat topology and the Kummer log \'etale topology on $(\mr{fs}/S)$, see \cite[Def. 2.3]{kat2} or \cite[Def. 2.13]{niz1}. 

\begin{defn}\label{defn1.1}
For $U\in (\mr{fs}/S)$, a family of morphisms $\{U_i\xrightarrow{f_i} U\}$ is called a \textbf{Kummer log flat cover} (resp. \textbf{Kummer log \'etale cover}), if the following are satisfied.
\begin{enumerate}[(1)]
\item Each $f_i$ is log flat (resp. log \'etale) and of Kummer type, and its underlying map of schemes is locally of finite presentation.
\item The family is set-theoretically surjective, i.e. $U=\bigcup_if_i(U_i)$.
\end{enumerate}
The \textbf{Kummer log flat topology} (resp. \textbf{Kummer log \'etale topology}) on $(\mr{fs}/S)$ is the Grothendieck topology given by the Kummer log flat cover (resp. Kummer log \'etale cover) on $(\mr{fs}/S)$. We will sometimes call the Kummer log flat (resp. Kummer log \'etale) topology simply the Kummer flat (resp. Kummer \'etale) topology, and denote the resulting site by $(\mr{fs}/S)_{\mr{kfl}}$ (resp. $(\mr{fs}/S)_{\mr{k\acute{e}t}}$).
\end{defn}

To see that these are indeed Grothendieck topology, we refer to \cite[\S 2]{kat2} and \cite[\S 2]{niz1}. 

By taking the strict flat covers, i.e. families of strict morphisms whose underlying maps of schemes form flat covers of schemes, one gets the classical flat site on $(\mr{fs}/S)$, denoted as $(\mr{fs}/S)_{\mr{fl}}$. Similarly, one also gets the classical \'etale site on $(\mr{fs}/S)$, denoted as $(\mr{fs}/S)_{\mr{\acute{e}t}}$. We have a natural ``forgetful'' map of sites
$$\varepsilon:(\mr{fs}/S)_{\mr{kfl}}\rightarrow (\mr{fs}/S)_{\mr{fl}}.$$
There is of course the \'etale version of the above forgetful map of sites, but we do not need it in this article. In order to shorten formulas, we will mostly abbreviate $\fsSet$ (resp. $\fsSket$, resp. $\fsSfl$, resp. $\fsSkfl$) as $\Set$ (resp. $\Sket$, resp. $\Sfl$, resp. $\Skfl$).

\begin{defn}\label{defn1.2}
Kato's multiplicative group (or the log multiplicative group) $\Gml$ is the sheaf on $\Set$ defined by $\Gml(U)=\Gamma(U,M^{\mr{gp}}_U)$
for any $U\in\fsS$, where $M_U$ denotes the log structure of $U$ and $M^{\mr{gp}}_U$ denotes the group envelope of $M_U$. 
\end{defn}

The \'etale sheaf $\Gml$ is also a sheaf on $\Skfl$, see \cite[Thm. 3.2]{kat2} or \cite[Cor. 2.22]{niz1}.

By convention, for any sheaf of abelian groups $F$ on $\Skfl$ and a subgroup sheaf $G$ of $F$ on $\Skfl$, we denote by $(F/G)_{\Set}$ (resp. $(F/G)_{\Sfl}$, resp. $(F/G)_{\Sket}$) the quotient sheaf on $\Set$ (resp. $\Sfl$, resp. $\Sket$), while $F/G$ denotes the quotient sheaf on $\Skfl$.

\subsection{Finite Kummer flat group log schemes and log $p$-divisible groups}

\begin{defn}\label{defn1.3}
The category $(\mr{fin}/S)_{\mr{c}}$ is the full subcategory of the category of sheaves of abelian groups over $(\mr{fs}/S)_{\mr{kfl}}$ consisting of objects which are representable by a classical finite flat group scheme over $S$. Here classical means that the log structure of the representing log scheme is induced from $S$.

The category $(\mr{fin}/S)_{\mr{f}}$ is the full subcategory of the category of sheaves of abelian groups over $(\mr{fs}/S)_{\mr{kfl}}$ consisting of objects which are representable by a classical finite flat group scheme over a Kummer log flat cover of $S$. For $F\in (\mr{fin}/S)_{\mr{f}}$, let $U\rightarrow S$ be a Kummer log flat cover of $S$ such that $F_U:=F\times_S U\in (\mr{fin}/U)_{\mr{c}}$. Then the rank of $F$ is defined to be  the rank of $F_U$ over $U$.

The category $(\mr{fin}/S)_{\mr{r}}$ is the full subcategory of $(\mr{fin}/S)_{\mr{f}}$ consisting of objects which are representable by a log scheme over $S$.

Let $F\in (\mr{fin}/S)_{\mr{f}}$, the Cartier dual of $F$ is the sheaf $F^*:=\mc{H}om_{S_{\mr{kfl}}}(F,\Gm)$. By the definition of $(\mr{fin}/S)_{\mr{f}}$, it is clear that $F^*\in (\mr{fin}/S)_{\mr{f}}$.

The category $(\mr{fin}/S)_{\mr{d}}$ is the full subcategory of $(\mr{fin}/S)_{\mr{r}}$ consisting of objects whose Cartier dual also lies in $(\mr{fin}/S)_{\mr{r}}$.
\end{defn}
The category $(\mr{fin}/S)_{\mr{r}}$ also has an alternative description as follows.

\begin{prop}\label{prop1.1}
Let $G$ be a sheaf of abelian groups on $(\mr{fs}/S)_{\mr{kfl}}$. Then we have $G\in (\mr{fin}/S)_{\mr{r}}$ if and only if $G$ satisfies the following condition.
\begin{enumerate}[($\star$)]
\item $G$ is representable by an fs log scheme such that the structure morphism $G\to S$ is Kummer log flat and its underlying morphism of schemes is finite.
\end{enumerate} 
\end{prop}
\begin{proof}
Assume that $G$ satisfied the condition ($\star$). We want to show that $G\in (\mr{fin}/S)_{\mr{r}}$. As the problem is classically \'etale local on $S$, we can assume that $S$ is quasi-compact and that the log structure on $S$ admits a global chart. Moreover, if $S$ is quasi-compact, so is $G$. Hence by \cite[Theorem 2.7 (2)]{kat2} there is a Kummer log flat cover $S'\to S$ such that $G_{S'}\to S'$ is strict. By \cite[Lem. 4.3.1]{k-s1}, the morphism $G_{S'}\to S'$ is classically flat. The underlying morphism of schemes of $G_{S'}\to S'$ is also finite by \cite[1.10]{nak1}. Hence $G_{S'}\to S'$ is a classical finite flat group scheme over $S'$. It follows that $G\in (\mr{fin}/S)_{\mr{r}}$.

Conversely assume that $G\in (\mr{fin}/S)_{\mr{r}}$, i.e. there is a Kummer log flat cover $S'\to S$ such that $G_{S'}\to S'$ is a classical finite flat group scheme. The property of being log-flat descends along Kummer log flat covers of the base, by \cite[Thm. 0.1]{i-n-t1}. Hence $G$ is log flat over $S$. Moreover by \cite[Prop. 2.7 (1)]{kat2} $G\to S$ is of Kummer type. To finish the proof, we are left with showing that the underlying map of schemes of $G\to S$ is finite, which follows from Proposition \ref{propC.1}\footnote{This is actually contained in \cite[Thm. 7.1]{kat2}. However the proof for the descent of finiteness there is referred to Tani's thesis \cite{tan1} which is in Japanese. So we present a proof in Proposition \ref{propC.1}.}. 
\end{proof}

\begin{defn}
Due to Proposition \ref{prop1.1}, we call an object of $(\mr{fin}/S)_{\mr{r}}$ a \textbf{finite Kummer log flat group log scheme}, or simply \textbf{finite Kummer flat group log scheme}, or even a \textbf{finite kfl group log scheme}.
\end{defn}

\begin{defn}\label{defn1.5}
Let $p$ be a prime number. A log $p$-divisible group over $S$ is a sheaf of abelian groups $G$ on $(\mr{fs}/S)_{\mr{kfl}}$ satisfying:
\begin{enumerate}[(1)]
\item $G=\varinjlim_{n}G_n$ with $G_n:=\mr{ker}(p^n:G\rightarrow G)$;
\item $p:G\rightarrow G$ is surjective;
\item $G_n\in (\mr{fin}/S)_{\mr{r}}$ for any $n> 0$.
\end{enumerate}
We denote the category of log $p$-divisible groups over $S$ by $(\text{$p$-div}/S)^{\mr{log}}_{\mr{r}}$. The full subcategory of $(\text{$p$-div}/S)^{\mr{log}}_{\mr{r}}$ consisting of objects $G$ with $G_n\in (\mr{fin}/S)_{\mr{d}}$ for $n>0$ will be denoted by $(\text{$p$-div}/S)^{\mr{log}}_{\mr{d}}$. A log $p$-divisible group $G$ with $G_n\in (\mr{fin}/S)_{\mr{c}}$ for $n>0$ is clearly just a classical $p$-divisible group, and we denote the full subcategory of $(\text{$p$-div}/S)^{\mr{log}}_{\mr{d}}$ consisting of classical $p$-divisible groups by $(\text{$p$-div}/S)^{\mr{log}}_{\mr{c}}$.
\end{defn}

\begin{rmk}
In Definition \ref{defn1.5}, one can replace the condition (3) by simply requiring $G_1\in (\mr{fin}/S)_{\mr{r}}$. This follows from the fact that the category $(\mr{fin}/S)_{\mr{r}}$ is closed under extension in the category of sheaves of abelian groups on $(\mr{fs}/S)_{\mr{kfl}}$, see \cite[Prop. A.1]{zha3}. Similarly, a log $p$-divisible group $G$ lies in $(\text{$p$-div}/S)^{\mr{log}}_{\mr{d}}$ if $G_1\in (\mr{fin}/S)_{\mr{d}}$, this also follows from \cite[Prop. A.1]{zha3}. Note that the corresponding statement for $(\text{$p$-div}/S)^{\mr{log}}_{\mr{c}}$ does not hold.
\end{rmk}

\section{Kato's classification theorem of log $p$-divisible groups}

\subsection{Standard extensions of log finite flat group objects}\label{subsec2.1}
In this subsection, we further assume that $S$ admits a global chart $P\rightarrow M_S$ with $P$ an fs monoid. Throughout the paper, if $Q$ is a monoid, we also denote by $Q$ the constant sheaf associated to $Q$.

Let $F''\in (\mr{fin}/S)_{\mr{c}}$ be \'etale, let $F'\in (\mr{fin}/S)_{\mr{c}}$, and let $n$ be a positive integer which kills both $F'$ and $F''$. We denote by $$\mathfrak{Ext}_{S_{\mr{kfl}}}(F'',F')\quad (\text{resp. }\mathfrak{Ext}_{S_{\mr{fl}}}(F'',F'))$$
the category of extensions of $F''$ by $F'$ in $(\mr{fin}/S)_{\mr{r}}$ (resp. $(\mr{fin}/S)_{\mr{c}}$). Let $F''(1):=F''\otimes_{\Z/n\Z}\Z/n\Z(1)$, and we denote by 
$$\mathfrak{Hom}(F''(1),F')\otimes_{\Z}P^{\mr{gp}}$$
the discrete category associated to the set $\mr{Hom}_S(F''(1),F')\otimes_{\Z}P^{\mr{gp}}$. There is a natural functor
\begin{equation}\label{eq2.1}
\Phi:\mathfrak{Ext}_{S_{\mr{fl}}}(F'',F')\times \mathfrak{Hom}(F''(1),F')\otimes_{\Z}P^{\mr{gp}}\rightarrow \mathfrak{Ext}_{S_{\mr{kfl}}}(F'',F')
\end{equation}
constructed as follows.

Firstly we construct a functor 
\begin{equation}\label{eq2.2}
\Phi_2:\mathfrak{Hom}(F''(1),F')\otimes_{\Z}P^{\mr{gp}}\rightarrow \mathfrak{Ext}_{S_{\mr{kfl}}}(F'',F').
\end{equation}
For every element $a\in P^{\mr{gp}}$, let $\mathbf{M}_a$ be the log 1-motive $[\Z\xrightarrow{1\mapsto a}\Gml]$. Then $E_a:=H^{-1}(\mathbf{M}_a\otimes^{L}_{\Z}\Z/n\Z)$ fits into a short exact sequence $$0\rightarrow\Z/n\Z(1)\rightarrow E_a\rightarrow\Z/n\Z\rightarrow0,$$
which splits Kummer flat locally. Hence we get another short exact sequence 
$$0\rightarrow F''(1)\rightarrow E_a\otimes_{\Z/n\Z} F''\rightarrow F''\rightarrow0$$
after tensoring with $F''$. For any $N\in\mr{Hom}_S(F''(1),F')$, we define $\Phi_2(N\otimes a)$ as the push-out
\begin{equation}\label{def-phi2}
\xymatrix{
0\ar[r] &F''(1)\ar[r]\ar[d]^N &E_a\otimes_{\Z/n\Z} F''\ar[r]\ar[d] &F''\ar[r]\ar@{=}[d] &0 \\
0\ar[r] &F'\ar[r] &N_*(E_a\otimes_{\Z/n\Z} F'')\ar[r] &F''\ar[r] &0
}.
\end{equation}
Now for any $\beta=\sum_iN_i\otimes a_i\in \mr{Hom}_S(F''(1),F')\otimes_{\Z}P^{\mr{gp}}$, we define $\Phi_2(\beta)\in\mathfrak{Ext}_{S_{\mr{kfl}}}(F'',F')$ as the Baer sum of the extensions $\Phi_2(N_i\otimes a_i)$. Now the functor $\Phi$ is defined as
$$\Phi(F^{\mr{cl}},\beta):=F^{\mr{cl}}+_{\mr{Baer}}\Phi_2(\beta)$$
for $F^{\mr{cl}}\in \mathfrak{Ext}_{S_{\mr{fl}}}(F'',F')$ and $\beta\in \mr{Hom}_S(F''(1),F')\otimes_{\Z}P^{\mr{gp}}$. Here the sum $+_{\mr{Baer}}$ denotes the Baer sum.

\begin{defn}\label{defn2.1}
An extension $F$ of $F''$ by $F'$ in the category $(\mr{fin}/S)_{\mr{r}}$ is called \textbf{standard with respect to the given chart}, if it lies in the essential image of the functor $\Phi$. We denote by $\mathfrak{Ext}_{S_{\mr{kfl}}}(F'',F')_{\mr{std}}$ the full subcategory of $\mathfrak{Ext}_{S_{\mr{kfl}}}(F'',F')$ consisting of standard extensions. Then the functor $\Phi$ induces a functor
\begin{equation}\label{eq2.3}
\Phi_{\mr{std}}:\mathfrak{Ext}_{S_{\mr{fl}}}(F'',F')\times \mathfrak{Hom}(F''(1),F')\otimes_{\Z}P^{\mr{gp}}\rightarrow \mathfrak{Ext}_{S_{\mr{kfl}}}(F'',F')_{\mr{std}}.
\end{equation}
For any $F=\Phi(F^{\mr{cl}},\beta)$ in $\mathfrak{Ext}_{S_{\mr{kfl}}}(F'',F')_{\mr{std}}$, we call $\beta$ a \textbf{pseudo-monodromy} of the extension $F$.  If the functor above is an equivalence of categories, we call $\beta$ a \textbf{monodromy} of $F$.
\end{defn}

In general, not every extension is standard, and a standard extension $F$ admits more than one pseudo-monodromy. The name pseudo-monodromy comes from the fact, that for arbitrary charts, the decompositions one gets might be strange. In particular a purely classical group scheme might have non-zero pseudo-momodromy. 

Now we are going to investigate the essential image of the functor $\Phi$.

Given an element $0\rightarrow F'\rightarrow F\rightarrow F''\rightarrow0$ in $\mathfrak{Ext}_{S_{\mr{kfl}}}(F'',F')$, applying the direct image functor $\varepsilon_*$, we get a long exact sequence
\begin{equation}\label{eq2.4}
0\rightarrow F'\rightarrow F\rightarrow F''\xrightarrow{\delta_F} R^1\varepsilon_* F'.
\end{equation}
Note that if $F=\Phi(F^{\mr{cl}}, \beta)$, we have $\delta_{F^{\mr{cl}}}=0$, so the connecting map $\delta_{F}$ only depends on $\beta$.
By \cite[Thm. 3.12]{niz1} or \cite[Thm.4.1]{kat2}, we have 
\[R^1\varepsilon_* F'=W\otimes_{\Z}(\Gml/\Gm)_{S_{\mr{fl}}}\] with 
\[W:=\varinjlim_r\mc{H}om_{S}(\Z/r\Z(1),F').\]
Note that $W=\mc{H}om_{S}(\Z/n\Z(1),F')$, since $F'$ is $n$-torsion. Then the homomorphism $\delta_F$ can be rewritten as $\delta_F: F''\rightarrow W\otimes_{\Z}(\Gml/\Gm)_{S_{\mr{fl}}}$, which is an element of the group $\mr{Hom}_S(F'',W\otimes_{\Z}(\Gml/\Gm)_{S_{\mr{fl}}})$. At the same time, the chart $P\rightarrow M_S$ induces a canonical homomorphism $\gamma_{P}:P^{\mr{gp}}\rightarrow (\Gml/\Gm)_{S_{\mr{fl}}}$, where $P^{\gp}$ denotes the group envelope of the monoid $P$. Therefore, we have a canonical homomorphism 
\begin{equation}\label{map}
\mr{id}_{W}\otimes \gamma_{P}:W\otimes_{\Z} P^{\gp}\rightarrow W\otimes_{\Z}(\Gml/\Gm)_{S_{\mr{fl}}}=R^1\varepsilon_*F'.
\end{equation}

\begin{rmk}
In \cite[\S 3.7]{kat4}, 
$$R^1\varepsilon_*F'=\varinjlim_r\mc{H}om_S(\Z/r\Z(1),F')\otimes (\Gml/\Gm)_{S_{\mr{fl}}}$$
is simply replaced by 
$$\varinjlim_r\mc{H}om_S(\Z/r\Z(1),F')\otimes P^{\mr{gp}}$$ 
without further explanation. From the authors' point of view, this is not completely correct, as $(\Gml/\Gm)_{S_{\mr{fl}}}$ is not constant, even in the henselian local case. This is one reason why we present a complete treatment to Kato's Theorem \ref{thm2.1}.
\end{rmk}
We have a canonical identification
\begin{equation}\label{ident}
\mr{Hom}_S(F'',W\otimes_{\Z}P^{\gp})=\mr{Hom}_S(F'',W)\otimes_{\Z}P^{\mr{gp}}=\mr{Hom}_S(F''(1),F')\otimes_{\Z}P^{\mr{gp}}.
\end{equation}
\begin{lem}\label{lem2.1}
Let $\beta\in \mr{Hom}_S(F''(1),F')\otimes_{\Z}P^{\mr{gp}}$, and let $F_{\beta}:=\Phi_2(\beta)$. Applying the functor $\varepsilon_*$ to the short exact sequence $0\rightarrow F'\rightarrow F_{\beta}\rightarrow F''\rightarrow0$, gives rise to a long exact sequence $0\rightarrow F'\rightarrow F_{\beta}\rightarrow F''\xrightarrow{\delta_{F_{\beta}}}R^1\varepsilon_*F'$. Let $\bar{\delta}_{F_{\beta}}$ be the homomorphism corresponding to $\beta$ under the canonical identification (\ref{ident}).
Then the homomorphism $\delta_{F_{\beta}}$ factors as 
$$\xymatrix{
&W\otimes_{\Z} P^{\gp}\ar[d]^{\mr{id}_{W}\otimes \gamma_{P}}  \\
F''\ar[r]^{\delta_{F_{\beta}}}\ar@{-->}[ru]^{\bar{\delta}_{F_{\beta}}} &R^1\varepsilon_*F'
}.$$
\end{lem}
\begin{proof}
It suffices to consider the case $\beta=N\otimes a\in\mr{Hom}_S(F''(1),F')\otimes_{\Z}P^{\mr{gp}}$. 

We abbreviate $(\Gml/\Gm)_{S_{\mr{fl}}}$ as $\mc{G}$ in order to shorten formulas. Applying the functor $\varepsilon_*$ to the short exact sequence 
\[0\rightarrow\Z/n\Z(1)\rightarrow E_a\rightarrow\Z/n\Z\rightarrow0,\]
we get a map
$$\delta_{E_{a}}:\Z/n\Z\rightarrow R^1\varepsilon_*\Z/n\Z(1)=\mc{H}om_S(\Z/n\Z(1),\Z/n\Z(1))\otimes_{\Z}\mc{G},\bar{1}\mapsto \bar{1}\otimes a.$$
Applying the functor $\varepsilon_*$ to $0\rightarrow F''(1)\rightarrow E_a\otimes_{\Z/n\Z} F''\rightarrow F''\rightarrow0$, we get another map
$$\delta_{E_{a}\otimes_{\Z/n\Z} F''} :F''=\Z/n\Z\otimes_{\Z/n\Z}F''\rightarrow R^1\varepsilon_*F''(1)=R^1\varepsilon_*\Z/n\Z(1)\otimes_{\Z/n\Z}F''.$$
The map $\delta_{E_{a}\otimes_{\Z/n\Z} F''}$ is identical to the map $\delta_{E_{a}}\otimes 1_{F''}$. Applying the functor $\varepsilon_*$ to the commutative diagram (\ref{def-phi2}), we get the following commutative diagram
$$\xymatrix{
F''\ar[rr]^-{\delta_{E_{a}\otimes_{\Z/n\Z} F''}}\ar@{=}[d] &&R^1\varepsilon_*F''(1)\ar[d]^{R^1\varepsilon_*N}\ar@{=}[r] &\mc{H}om_S(\Z/n\Z(1),F''(1))\otimes_{\Z}\mc{G}\ar[d]^{N_*\otimes\mr{id}_{\mc{G}}}\\
F''\ar[rr]^-{\delta_{F_{\beta}}} &&R^1\varepsilon_*F'\ar@{=}[r] &\mc{H}om_S(\Z/n\Z(1),F')\otimes_{\Z}\mc{G}  
}.$$
It is clear that $\delta_{E_{a}}$ factors through 
$$\bar{\delta}_{E_{a}}:\Z/n\Z\rightarrow \mc{H}om_S(\Z/n\Z(1),\Z/n\Z(1))\otimes_{\Z}P^{\gp},\,\bar{1}\mapsto \bar{1}\otimes a .$$
Therefore $\delta_{E_{a}\otimes_{\Z/n\Z} F''}$ factors through 
$$\bar{\delta}_{E_{a}\otimes_{\Z/n\Z} F''}:F''=\Z/n\Z\otimes_{\Z/n\Z}F''\rightarrow \mc{H}om_S(\Z/n\Z(1),\Z/n\Z(1))\otimes_{\Z}P^{\gp}\otimes_{\Z/n\Z}F'',$$
where $\bar{\delta}_{E_{a}\otimes_{\Z/n\Z} F''}$ is defined as the map $\bar{\delta}_{E_{a}}\otimes 1_{F''}$. We can rewrite $\bar{\delta}_{E_{a}\otimes_{\Z/n\Z} F''}$ as 
\[F''\rightarrow\mc{H}om_S(\Z/n\Z(1),F''(1))\otimes_{\Z}P^{\gp}.\] 
We have an obvious commutative diagram
$$\xymatrix{
\mc{H}om_S(\Z/n\Z(1),F''(1))\otimes_{\Z}P^{\gp}\ar[r]\ar[d]^{N_*\otimes\mr{id}_{P^{\gp}}} &R^1\varepsilon_*F''(1)\ar[d]^{R^1\varepsilon_*N} \\
\mc{H}om_S(\Z/n\Z(1),F')\otimes_{\Z}P^{\gp}\ar[r] &R^1\varepsilon_*F'
}.$$
It follows that $\delta_{F_{\beta}}$ factors through $\bar{\delta}_{F_{\beta}}:=(N_*\otimes\mr{id}_{P^{\gp}})\circ\bar{\delta}_{E_{a}\otimes_{\Z/n\Z} F''}$. This finishes the proof.
\end{proof}

\begin{rmk}
In the proof of Lemma \ref{lem2.1}, the map 
$$\bar{\delta}_{E_{a}\otimes_{\Z/n\Z} F''}:F''=F''\otimes_{\Z}\Z\rightarrow \mc{H}om_S(\Z/n\Z(1),F''(1))\otimes_{\Z}P^{\gp}$$
corresponds to 
\[\mr{id}_{F''(1)}\in\mr{Hom}_S(F'',\mc{H}om_S(\Z/n\Z(1),F''(1)))=\mr{Hom}_S(F''(1),F''(1))\]
and $\Z\rightarrow P^{\mr{gp}},1\mapsto a$. Therefore $(N_*\otimes\mr{id}_{P^{\gp}})\circ\bar{\delta}_{E_{a}\otimes_{\Z/n\Z} F''}$ is nothing but 
$$N\otimes_{\Z}a\in \mr{Hom}_S(F''(1),F')\otimes_{\Z}P^{\mr{gp}}.$$
\end{rmk}

\begin{prop}\label{prop2.1}
Let $F$ be an extension of $F''$ by $F'$ in the category $(\mr{fin}/S)_{\mr{r}}$. Then  $F$ is a standard extension with respect to the given chart if and only if the homomorphism $\delta_F$ from (\ref{eq2.4}) factors 
$$\xymatrix{
&&&&W\otimes P^{\mr{gp}}\ar[d]  \\
0\ar[r] &F'\ar[r] &F\ar[r] &F''\ar[r]^{\delta_F}\ar@{-->}[ru]^{\bar{\delta}_F} &R^1\varepsilon_*F'
}$$
through the canonical homomorphism $W\otimes P^{\mr{gp}}\rightarrow W\otimes_{\Z}(\Gml/\Gm)_{S_{\mr{fl}}}=R^1\varepsilon_*F'$.
\end{prop}
\begin{proof}
The ``only-if'' part follows from Lemma \ref{lem2.1} easily. Now we assume that $\delta_F$ has a factorization as above. We want to express $F$ as a standard extension. 

First we digress to investigate the behavior of the map (\ref{eq2.5}) with respect to Baer sums. Let $F_1$, $F_2$ be two extensions of $F''$ by $F'$. The following commutative diagram
\[\xymatrix{0\ar[r] &F'\oplus F'\ar[r]\ar[d]_{+} &F_1\oplus F_2\ar[r]\ar[d] &F''\oplus F''\ar[r]^-{\delta_{F_1}\oplus\delta_{F_2}}\ar@{=}[d] &R^1\varepsilon_*F'\oplus R^1\varepsilon_*F''\ar[d]^{+} \\
0\ar[r] &F'\ar[r] &+_*(F_1\oplus F_2)\ar[r] &F''\oplus F''\ar[r] &R^1\varepsilon_*F' \\
0\ar[r] &F'\ar[r]\ar@{=}[u] &F_1+_{\mr{Baer}} F_2\ar[r]\ar[u] &F''\ar[r]^-{\delta_{F_1+_{\mr{Baer}}F_2}}\ar[u]_{\mr{diagonal}} &R^1\varepsilon_*F'\ar@{=}[u] \\
}\] 
implies that $\delta_{F_1+_{\mr{Baer}}F_2}=\delta_{F_1}+\delta_{F_2}$, i.e. the formation of $\delta_{-}$ is compatible with Baer sum.

Now we go back to the proof of the proposition. By abuse of notation, we use the same notation $\bar{\delta}_{F}$ for the element in $\mr{Hom}_S(F''(1),F')\otimes_{\Z}P^{\mr{gp}}$ corresponding to $\bar{\delta}_{F}$ in $\mr{Hom}_S(F'',W\otimes_{\Z}P^{\mr{gp}})$. Let $F_{\bar{\delta}_F}:=\Phi_2(\bar{\delta}_F)$, and let $F^{\mr{cl}}:=F-_{\mr{Baer}}F_{\bar{\delta}_F}$. By Lemma \ref{lem2.1}, the connection map $F''\rightarrow R^1\varepsilon_*F'$ for $F^{\mr{cl}}$ is $\delta_F-\delta_F=0$. Therefore $F^{\mr{cl}}$ lies in $\mathfrak{Ext}_{S_{\mr{fl}}}(F'',F')$, and $F=F^{\mr{cl}}+_{\mr{Baer}}F_{\bar{\delta}_F}=F^{\mr{cl}}+_{\mr{Baer}}\Phi_2(\bar{\delta}_F)$ is standard with respect to the given chart.
\end{proof}

By Proposition \ref{prop2.1}, it is clear that the homomorphism
\begin{equation}\label{alpha}
\mr{Hom}_S(F'',W\otimes_{\Z}P^{\mr{gp}})\xrightarrow{\alpha} \mr{Hom}_S(F'',W\otimes_{\Z}(\Gml/\Gm)_{S_{\mr{fl}}})
\end{equation}
induced by $\gamma_{P}:P^{\mr{gp}}\rightarrow (\Gml/\Gm)_{S_{\mr{fl}}}$ is  important for understanding standard extensions.

For any $U\in(\mr{fs}/S)$, both $P^{\mr{gp}}$ and $(\Gml/\Gm)_{S_{\mr{fl}}}$ are constructible $\Z$-modules on the small \'etale site of $U$ (see \cite[Lemma 3.5 (ii)]{ols1}). By \cite[Chap. V, Rmk. 1.7 (e)]{mil1} and Lemma \ref{lemA.1}, the sheaf $W=\mc{H}om_S(\Z/n\Z(1),F')$ is a constructible $\Z$-module for the \'etale topology. By \cite[\href{https://stacks.math.columbia.edu/tag/095I}{Tag 095I}]{stacks-project}, the tensor product of $W$ and $P^{\mr{gp}}$ for the \'etale topology is a constructible $\Z$-module, hence it is representable by an algebraic space by \cite[Exp. IX, Prop. 2.7]{sga4}, therefore automatically a sheaf on $(\mr{fs}/S)_{\mr{fl}}$. It follows that the tensor product $W\otimes_{\Z}P^{\mr{gp}}$ on $(\mr{fs}/S)_{\mr{fl}}$ coincides with the corresponding tensor product for the \'etale topology. So we do not have to distinguish for which topology we take the tensor product of $W$ and $P^{\mr{gp}}$, and just simply write $W\otimes_{\Z}P^{\mr{gp}}$. We have similar result for $W\otimes_{\Z}(\Gml/\Gm)_{S_{\mr{fl}}}$.

Recall that the log structure in this paper is defined in the \'etale topology. Let $M_S^{\mr{gp}}/\mc{O}_S^\times$ be the quotient on the small \'etale site of $S$. Since $F''$ is \'etale locally constant, the homomorphism $\alpha$ is identified with 
\[\mr{Hom}_S(F'',W\otimes_{\Z}P^{\mr{gp}})\xrightarrow{\alpha} \mr{Hom}_S(F'',W\otimes_{\Z}M_S^{\mr{gp}}/\mc{O}_S^\times )\]
in which the Hom-groups are taken in the small \'etale site of $S$. By abuse of notation, we denote this homomorphism also by $\alpha$.

The chart $P\rightarrow M_S$ induces a canonical epimorphism $P^{\mr{gp}}\rightarrow M_S^{\mr{gp}}/\mc{O}_S^\times$. Let $K$ be the kernel of this homomorphism. We then get a short exact sequence
$$0\rightarrow K\rightarrow P^{\mr{gp}}\rightarrow M_S^{\mr{gp}}/\mc{O}_S^\times\rightarrow 0$$
of sheaves on the small \'etale site of $S$. Since the stalks of $K$, $P^{\mr{gp}}$ and $M_S^{\mr{gp}}/\mc{O}_S^\times$ are free abelian groups, the canonical sequence 
\begin{equation}\label{eq2.5}
0\rightarrow W\otimes_{\Z}K\rightarrow W\otimes_{\Z}P^{\mr{gp}}\rightarrow W\otimes_{\Z}M_S^{\mr{gp}}/\mc{O}_S^\times\rightarrow 0
\end{equation}
is exact on all stalks at geometric points of $S$. It follows that the sequence itself is exact. 

Now applying the functor $\mr{Hom}_S(F'',-)$ to (\ref{eq2.5}), we get another exact sequence
\begin{equation}\label{eq2.6}
0\rightarrow \mr{Hom}_S(F'',W\otimes_{\Z}K)\rightarrow \mr{Hom}_S(F'',W\otimes_{\Z}P^{\mr{gp}})\xrightarrow{\alpha} \mr{Hom}_S(F'',W\otimes_{\Z}M_S^{\mr{gp}}/\mc{O}_S^\times).
\end{equation}

\begin{prop}\label{prop2.2}
Let the notation and the assumptions be as above.
\begin{enumerate}[(1)]
\item Assume that $\mr{Hom}_S(F'',W\otimes_{\Z}K)=0$, i.e. the map $\alpha$ from (\ref{eq2.6}) is injective. Then the functor (\ref{eq2.3})
\[\Phi_{\mr{std}}:\mathfrak{Ext}_{S_{\mr{fl}}}(F'',F')\times \mathfrak{Hom}(F''(1),F')\otimes_{\Z}P^{\mr{gp}}\rightarrow \mathfrak{Ext}_{S_{\mr{kfl}}}(F'',F')_{\mr{std}}\] is an equivalence of categories.
  
\item Assume that $\alpha$ is an isomorphism. Then 
$$\mathfrak{Ext}_{S_{\mr{kfl}}}(F'',F')_{\mr{std}}=\mathfrak{Ext}_{S_{\mr{kfl}}}(F'',F'),$$
whence the functor $\Phi$ (\ref{eq2.1}) is an equivalence of categories.
\end{enumerate}
\end{prop}
\begin{proof}
We claim that $\Phi_{\mr{std}}$ admits a quasi-inverse, given by 
\begin{equation}\label{eq2.7}
\begin{split}
\Psi_{\mr{std}}:\mathfrak{Ext}_{S_{\mr{kfl}}}(F'',F')_{\mr{std}}&\rightarrow \mathfrak{Ext}_{S_{\mr{fl}}}(F'',F')\times \mathfrak{Hom}(F''(1),F')\otimes_{\Z}P^{\mr{gp}}  \\
F&\mapsto (F-_{\mr{Baer}}\Phi_2(\beta_F),\beta_F)
\end{split},
\end{equation}
where $\beta_F$, under the identification 
$$\mr{Hom}_S(F''(1),F')\otimes_{\Z}P^{\mr{gp}}\cong \mr{Hom}_S(F'',\mc{H}om(\Z/n\Z(1),F')\otimes_{\Z}P^{\mr{gp}}),$$
corresponds to $\bar{\delta}_F$, which is the unique lift of the connecting map $\delta_F:F''\rightarrow R^1\varepsilon_*F'$ of $F$.
Note that the uniqueness of $\beta_F$, which amounts to the uniqueness of $\bar{\delta}_F$, is guaranteed by $\mr{Hom}_S(F'',W\otimes_{\Z}K)=0$.

Since the functor (\ref{eq2.3}) is clearly essentially surjective, it suffices to show that $\Psi_{\mr{std}}\circ\Phi_{\mr{std}}=1$. 

By Lemma \ref{lem2.1}, for $F=\Phi_2(\beta)$ with $\beta\in\mr{Hom}_S(F''(1),F')\otimes_{\Z}P^{\mr{gp}}$, we have $\beta_F=\beta$. Hence for any $F\in\mathfrak{Ext}_{S_{\mr{kfl}}}(F'',F')_{\mr{std}}$, we have 
$$\beta_{F-_{\mr{Baer}}\Phi_2(\beta_F)}=\beta_F-\beta_{\Phi_2(\beta_F)}=\beta_F-\beta_F=0.$$ 
This implies that the direct image of the short exact sequence $$0\rightarrow F'\rightarrow (F-_{\mr{Baer}}\Phi_2(\beta_F))\rightarrow F''\rightarrow0$$
remains exact, therefore $F-_{\mr{Baer}}\Phi_2(\beta_F)$ is indeed classical. We also have 
$$\Psi_{\mr{std}}(\Phi_{\mr{std}}(F^{\mr{cl}},\beta))=\Psi_{\mr{std}}(F^{\mr{cl}}+_{\mr{Baer}}\Phi_2(\beta))=(F^{\mr{cl}},\beta).$$
This proves part (1).

For (2), note that under the assumption that $\alpha$ is an isomorphism, every extension of $F''$ by $F'$ is standard with respect to the given chart by Proposition \ref{prop2.1}. So part (2) follows from part (1).
\end{proof}

\begin{lem}\label{lem2.2}
Assume that the underlying scheme of $S$ is $\Spec A$ with $A$ a noetherian henselian local ring, and $S$ admits a global chart $P\rightarrow M_S$ such that the induced map $P\rightarrow M_{S,\bar{s}}/\mc{O}_{S,\bar{s}}^{\times}$ is an isomorphism for the closed point $s$ of $S$. Then the canonical homomorphism $\alpha$ (\ref{alpha}) is an isomorphism, so the group $\mr{Hom}_S(F'',W\otimes_{\Z}(\Gml/\Gm)_{S_{\mr{fl}}})$ can canonically be identified with the group 
$$\mr{Hom}_S(F'',W\otimes_{\Z}P^{\mr{gp}})=\mr{Hom}_S(F'',W)\otimes_{\Z}P^{\mr{gp}}.$$
\end{lem}
\begin{proof}
It suffices to prove that the map $\alpha$ from (\ref{eq2.6}) is an isomorphism. Then the result follows from $P^{\mr{gp}}\cong\Z^r$ for some $r\in\N$.

We first deal with the case that the finite \'etale group scheme $F''$ is constant. Then for any sheaf $G$ on the small \'etale site of $S$, the group $\mr{Hom}_S(F'',G)$ is determined by the group $\mr{Hom}(F''(S),G(S))$. On the other hand, we have $\Gamma(S,G)=\Gamma(s,G|_s)$ by Gabber's theorem, see \cite[\href{https://stacks.math.columbia.edu/tag/09ZI}{Tag 09ZI}]{stacks-project}. This implies that the exact sequence 
\[0\rightarrow\Gamma(S,W\otimes_{\Z}K)\rightarrow\Gamma(S,W\otimes_{\Z}P^{\mr{gp}})\xrightarrow{\beta}\Gamma(S,W\otimes_{\Z}(M_S^{\mr{gp}}/\mc{O}_S^\times))\]
can be identified with the exact sequence
$$0\rightarrow\Gamma(s,(W\otimes_{\Z}K)|_s)\rightarrow\Gamma(s,(W\otimes_{\Z}P^{\mr{gp}})|_s)\rightarrow\Gamma(s,(W\otimes_{\Z}(M_S^{\mr{gp}}/\mc{O}_S^\times))|_s),$$
where $(-)|_s$ denotes the pullback from the small \'etale site of $S$ to that of $s$.
By the property of the given chart $P\rightarrow M_S$, we have $P^{\mr{gp}}|_s\cong (M_S^{\mr{gp}}/\mc{O}_S^\times)|_s$. It follows that the map $\beta$ is an isomorphism. Therefore the map $\alpha$ is an isomorphism.

In general, choose a finite \'etale cover $\tilde{S}=\Spec \tilde{A}\rightarrow\Spec A$ with $\tilde{A}$ henselian local, such that $F''|_{\tilde{S}}$ is a constant group scheme. Let $$f\in \mr{Hom}_S(F'',W\otimes_{\Z}(M_S^{\mr{gp}}/\mc{O}_S^\times)).$$
Then $\tilde{f}:=f\times_S\tilde{S}$ lifts to a unique $\tilde{g}\in \mr{Hom}_{\tilde{S}}(F'',W\otimes_{\Z}P^{\mr{gp}})$ by the previous case. Since $\tilde{S}\times_S\tilde{S}$ is a disjoint union of henselian local schemes which are finite \'etale over $\tilde{S}$, we have $p_1^*\tilde{g}=p_2^*\tilde{g}$ by the previous case again, where $p_1$ and $p_2$ are the projections from $\tilde{S}\times_S\tilde{S}$ to its two factors. It follows that $\tilde{g}$ descends to a unique element of $\mr{Hom}_S(F'',W\otimes_{\Z}P^{\mr{gp}})$. This finishes the proof.
\end{proof}

\begin{thm}[Kato]\label{thm2.1}
Let the assumptions be as in Lemma \ref{lem2.2}. Then
$$\mathfrak{Ext}_{S_{\mr{kfl}}}(F'',F')_{\mr{std}}=\mathfrak{Ext}_{S_{\mr{kfl}}}(F'',F')$$
and the functor (\ref{eq2.1})
$$\Phi:\mathfrak{Ext}_{S_{\mr{fl}}}(F'',F')\times \mathfrak{Hom}(F''(1),F')\otimes_{\Z}P^{\mr{gp}}\rightarrow \mathfrak{Ext}_{S_{\mr{kfl}}}(F'',F')$$
is an equivalence of categories with inverse (\ref{eq2.7})
$$\Psi:\mathfrak{Ext}_{S_{\mr{kfl}}}(F'',F')\rightarrow \mathfrak{Ext}_{S_{\mr{fl}}}(F'',F')\times \mathfrak{Hom}(F''(1),F')\otimes_{\Z}P^{\mr{gp}}.$$
\end{thm}
\begin{proof}
This follows from Lemma \ref{lem2.2} and Proposition \ref{prop2.2}.
\end{proof}

\subsection{The connected-\'etale sequence}

Assume that $S=\Spec(R)$ is a local henselian ring with residue characteristic $p$. The results in this section are all due to Kato and we will only sketch the proofs.

For any $G\in (\mr{fin}/S)_{\mr{f}}$ there is a unique exact sequence
\begin{center}
$0\to G^{\circ}\to G\to G^{\rm{\acute{e}t}}\to 0$
\end{center}
which restricts to the usual connected-\'etale sequence over any finite Kummer log flat cover, over which $G$ becomes a classical flat group scheme. To get this, one can descend the usual connected-\'etale sequence from some cover over which $G$ is classical (the sequence descends, because it is unique and compatible with base-change).

One now has the following result:
\begin{prop}\label{c-e}
Let $G\in(\mr{fin}/S)_{\mr{d}}$. Assume that the residue characteristic $p$ is positive, and $G$ is killed by a power of $p$. Then $G^{\circ}$ and $G^{\mr{\acute{e}t}}$ are classical finite flat group schemes.
\end{prop}
\begin{proof}
One has to show that a connected finite kfl log group scheme $f:H\to S$ is classical. For this, one looks at the unit map $e:S\to H$, and sees that for any geometric point $\bar{s}$ of $S$, one gets that the identity factors as
\begin{center}
$(M_{S}/\mathcal{O}^{\times}_{S})_{\bar{s}}\to (M_{H}/\mathcal{O}^{\times}_{H})_{e(\bar{s})}\to (M_{S}/\mathcal{O}^{\times}_{S})_{\bar{s}}$.
\end{center}
Using that $H\to S$ is Kummer, one concludes that $f^{-1}(M_{S}/\mathcal{O}^{\times}_{S})\to M_{H}/\mathcal{O}^{\times}_{H}$ is an isomorphism.

From this it then follows that $G^{\circ}$ is classical. As the dual $(G^{\mr{\acute{e}t}})^{D}$ of $G^{\mr{\acute{e}t}}$ is representable by assumption and clearly connected, it is classical. Hence $G^{\mr{\acute{e}t}}$ is classical as well.

\end{proof} 
Recall the following definition. Let $Z$ be an fs log scheme and $z$ a point of $Z$. A chart $P\to M_Z$ of $Z$ is called \textbf{neat} at $z$, if the induced map $P\xrightarrow{\cong }M_{Z,\bar{z}}/\mathcal{O}_{Z,\bar{z}}^{\times}$ is an isomorphism, see \cite[Chap. II, Def. 2.3.1]{ogu1}. Neat charts exist \'etale locally by \cite[Chap. II, Prop. 2.3.7]{ogu1}.

We also denote by $(\mr{fin}/S)^{p}_{\mr{d}}$ the subcategory of $(\mr{fin}/S)_{\mr{d}}$ whose objects are $p$-power torsion. Using theorem \ref{thm2.1} one then gets the following corollary.

\begin{cor}\label{cor2.1}
Assume that $P\to M_S$ is a neat chart and the residue characteristic $p$ is positive. Then there is an equivalence of categories between $(\mr{fin}/S)^{p}_{\mr{d}}$ and the category of pairs $(G^{\mr{cl}}, N)$, where $G^{\mr{cl}}$ is a classical extension of $G^{\mr{\acute{e}t}}$ by $G^{\circ}$, and $N\in \mr{Hom}(G^{\mr{\acute{e}t}}(1), G^{\circ})\otimes_{\Z}P^{\mr{gp}}$.
\end{cor}

\subsection{Extensions are always standard \'etale locally}
In the last section, we studied standard extensions of classical finite flat group schemes after fixing a chart on the base. In this section we show that any extension of one classical finite \'etale group scheme by another classical finite flat group scheme is always \'etale locally standard with respect to a suitable chosen local chart.

\begin{thm}\label{thm2.2}
Let $S$ be an fs log scheme whose underlying scheme is locally noetherian. Let $F',F''\in(\mr{fin}/S)_{\mr{c}}$ with $F''$ \'etale, $F\in(\mr{fin}/S)_{\mr{r}}$ an extension of $F''$ by $F'$. Then for any $s\in S$, there exists an \'etale neighborhood $\tilde{S}$ of $s$ such that the log structure of $S$ admits a chart $P_{\tilde{S}}\to M_{\tilde{S}}=(M_S)|_{\tilde{S}}$ on $\tilde{S}$ with $P$ an fs monoid, and $F$ is standard with respect to this chart.
\end{thm}
\begin{proof}
By shrinking $S$ \'etale locally if necessary, we assume that $S$ admits a chart $P\to M_S$ such that the induced map $P\to M_{S,\bar{s}}/\mc{O}_{S,\bar{s}}^\times$ is an isomorphism. We adopt the notation from Subsection \ref{subsec2.1}. By Proposition \ref{prop2.1}, we are left with lifting 
\[\delta_F:F''\rightarrow R^1\varepsilon_*F'=W\otimes_{\Z}(\Gml/\Gm)_{S_{\mr{fl}}}\]
to a homomorphism $\bar{\delta}_F:F''\to W\otimes_\Z P^{\mr{gp}}$ \'etale locally around $s$. 

The short exact sequence (\ref{eq2.6}) extends to a longer exact sequence
$$\mr{Hom}_S(F'',W\otimes_{\Z}P^{\mr{gp}})\xrightarrow{\alpha} \mr{Hom}_S(F'',W\otimes_{\Z}M_S^{\mr{gp}}/\mc{O}_S^\times)\xrightarrow{\gamma} \mr{Ext}^1_{S_{\mr{\acute{e}t}}}(F'',W\otimes_{\Z}K).$$
The local to global extension spectral sequence gives rise to an exact sequence
\begin{align*}
0\rightarrow &H^1(S_{\mr{\acute{e}t}},\mc{H}om_S(F'',W\otimes_{\Z}K))\rightarrow \mr{Ext}^1_{S_{\mr{\acute{e}t}}}(F'',W\otimes_{\Z}K)\rightarrow   \\
&\Gamma(S,\mc{E}xt^1_{S_{\mr{\acute{e}t}}}(F'',W\otimes_{\Z}K)).
\end{align*}

We claim that $\gamma(\delta_F)$ has trivial image in $\Gamma(S,\mc{E}xt^1_{S_{\mr{\acute{e}t}}}(F'',W\otimes_{\Z}K))$. It suffices to show that $\gamma(\delta_F)$ vanishes at all stalks of $\mc{E}xt^1_{S_{\mr{\acute{e}t}}}(F'',W\otimes_{\Z}K)$. Let $t$ be any point of $S$ with $\bar{t}$ a geometric point above $t$, and let $S_{\bar{t}}$ be the strict henselization of $S$ at $\bar{t}$. Pulling back the short exact sequence (\ref{eq2.5}) to $S_{\bar{t}}$ and then taking global section, we get a short exact sequence
$$0\rightarrow\Gamma(S_{\bar{t}},W\otimes_{\Z}K)\rightarrow \Gamma(S_{\bar{t}},W\otimes_{\Z}P^{\mr{gp}})\rightarrow \Gamma(S_{\bar{t}},W\otimes_{\Z}(M_S^{\mr{gp}}/\mc{O}_S^\times))\rightarrow 0.$$
By \cite[\href{https://stacks.math.columbia.edu/tag/09ZH}{Tag 09ZH}]{stacks-project}, this exact sequence can be identified with the short exact sequence
$$0\rightarrow\Gamma(\bar{t},(W\otimes_{\Z}K)|_{\bar{t}})\rightarrow \Gamma(\bar{t},(W\otimes_{\Z}P^{\mr{gp}})|_{\bar{t}})\rightarrow \Gamma(\bar{t},(W\otimes_{\Z}(M_S^{\mr{gp}}/\mc{O}_S^\times))|_{\bar{t}})\rightarrow 0.$$
Now over the geometric point $\bar{t}$, the epimorphism $P^{\mr{gp}}\to M_S^{\mr{gp}}/\mc{O}_S^\times$ admits a section. It follows that the epimorphism $\Gamma(S_{\bar{t}},W\otimes_{\Z}P^{\mr{gp}})\rightarrow \Gamma(S_{\bar{t}},W\otimes_{\Z}(M_S^{\mr{gp}}/\mc{O}_S^\times))$ also admits a section. This forces the homomorphism
$$\mr{Hom}_{S_{\bar{t}}}(F'',W\otimes_{\Z}P^{\mr{gp}})\xrightarrow{\alpha} \mr{Hom}_{S_{\bar{t}}}(F'',W\otimes_{\Z}(M_S^{\mr{gp}}/\mc{O}_S^\times))$$
to be surjective. It follows that the stalk of $\gamma(\delta_F)$ at $\bar{t}$ vanishes.

Now we see that $\gamma(\delta_F)$ lies in the subgroup $H^1(S,\mc{H}om_S(F'',W\otimes_{\Z}K))$ of $\mr{Ext}^1_{S_{\mr{\acute{e}t}}}(F'',W\otimes_{\Z}K)$. Hence passing to an \'etale cover of $S$, $\gamma(\delta_F)$ vanishes, in other words $\delta_F$ lifts to a homomorphism $\bar{\delta}_F:F''\to W\otimes_{\Z}P^{\mr{gp}}$ after passing to a suitable \'etale cover of $S$. This finishes the proof.
\end{proof}

\subsection{Standard extensions of log $p$-divisible groups}
First we assume that $S$ admits a global chart $P\rightarrow M_S$ with $P$ an fs monoid.

Let $H'=\varinjlim_{n}H'_n$, $H''=\varinjlim_{n}H''_n$ be two objects in $(\text{$p$-div}/S)^{\mr{log}}_{\mr{c}}$, and we assume that $H''$ is \'etale. We denote by $$\mathfrak{Ext}_{S_{\mr{kfl}}}(H'',H')\quad (\text{resp. }\mathfrak{Ext}_{S_{\mr{fl}}}(H'',H'))$$ the category of extensions of $H''$ by $H'$ in $(\text{$p$-div}/S)^{\mr{log}}_{\mr{r}}$ (resp. ($\text{$p$-div}/S)^{\mr{log}}_{\mr{c}}$). Let 
\[H''(1):=\varinjlim_n H''_n\otimes_{\Z/p^n\Z}\Z/p^n\Z(1),\]
and we denote by 
$$\mathfrak{Hom}(H''(1),H')\otimes_{\Z}P^{\mr{gp}}$$
the discrete category associated to the set $\mr{Hom}_S(H''(1),H')\otimes_{\Z}P^{\mr{gp}}$. Let 
\[H^{\mr{cl}}=\varinjlim_{n}H_n^{\mr{cl}}\in \mathfrak{Ext}_{S_{\mr{fl}}}(H'',H'),\]
and $\beta\in \mathfrak{Hom}(H''(1),H')\otimes_{\Z}P^{\mr{gp}}$. The element $\beta$ induces a compatible system 
\[\{\beta_n\in \mathfrak{Hom}(H''_n(1),H'_n)\otimes_{\Z}P^{\mr{gp}}\}_{n}.\]
We apply the functor (\ref{eq2.1}) to the pair $(H_n^{\mr{cl}},\beta_n)$ for each $n\geq 1$ and change the notation for the functor (\ref{eq2.1}) (resp. (\ref{eq2.2})) from $\Phi$ (resp. $\Phi_2$) to $\Phi^n$ (resp. $\Phi_2^n$) in order to indicate its relation with $(H_n^{\mr{cl}},\beta_n)$. Then we get a compatible system  $\{\Phi^n(H_n^{\mr{cl}},\beta_n)\}_n$ with 
$$\Phi^n(H_n^{\mr{cl}},\beta_n)=H_n^{\mr{cl}}+_{\mr{Baer}}\Phi_2^n(\beta_n)\in \mathfrak{Ext}_{S_{\mr{kfl}}}(H''_n,H'_n).$$
Note that $\Phi_2(\beta):=\varinjlim_n\Phi_2^n(\beta_n)$ is an object of $(\text{$p$-div}/S)^{\mr{log}}_{\mr{r}}$. Therefore 
$$\varinjlim_n\Phi^n(H_n^{\mr{cl}},\beta_n)=\varinjlim_n(H_n^{\mr{cl}}+_{\mr{Baer}}\Phi_2^n(\beta_n))$$
lies in $(\text{$p$-div}/S)^{\mr{log}}_{\mr{r}}$. We denote $\varinjlim_n\Phi^n(H_n^{\mr{cl}},\beta_n)$ by $\Phi(H^{\mr{cl}},\beta)$. The association of $\Phi(H^{\mr{cl}},\beta)$ to the pair $(H^{\mr{cl}},\beta)$ gives rise to a functor 
\begin{equation}\label{eq2.8}
\Phi: \mathfrak{Ext}_{S_{\mr{fl}}}(H'',H')\times \mathfrak{Hom}(H''(1),H')\otimes_{\Z}P^{\mr{gp}}\rightarrow \mathfrak{Ext}_{S_{\mr{kfl}}}(H'',H').
\end{equation}
Note that we use the same notation for both the functor (\ref{eq2.8}) and the functor (\ref{eq2.1}), but the potential confusion between the two functors can be cleared in the context.

\begin{defn}\label{defn2.2}
Let $H',H''\in (\text{$p$-div}/S)^{\mr{log}}_{\mr{c}}$ with $H''$ \'etale. An extension $H$ of $H''$ by $H'$ in the category $(\text{$p$-div}/S)^{\mr{log}}_{\mr{r}}$ is called \textbf{standard with respect to the given chart $P\rightarrow M_S$}, if it lies in the essential image of the functor (\ref{eq2.8}). We denote by $\mathfrak{Ext}_{S_{\mr{kfl}}}(H'',H')_{\mr{std}}$ the full subcategory of $\mathfrak{Ext}_{S_{\mr{kfl}}}(H'',H')$ consisting of the standard extensions, and the functor (\ref{eq2.8}) induces a functor
\begin{equation}\label{eq2.9}
\mathfrak{Ext}_{S_{\mr{fl}}}(H'',H')\times \mathfrak{Hom}(H''(1),H')\otimes_{\Z}P^{\mr{gp}}\rightarrow \mathfrak{Ext}_{S_{\mr{kfl}}}(H'',H')_{\mr{std}}.
\end{equation}
If $H\cong\Phi(H^{\mr{cl}},\beta)$, we call $\beta$ a \textbf{pseudo-monodromy of $H$ with respect to the given chart $P\rightarrow M_S$}. If the functor (\ref{eq2.9}) is an equivalence, the pseudo-monodromy $\beta$ is uniquely determined. In this case, we call $\beta$ the \textbf{monodromy of $H$ with respect to the given chart $P\to M_S$}.
\end{defn}

\begin{thm}[Kato]\label{thm2.3}
Let $S$ be an fs log scheme whose underlying scheme is $\Spec A$ for a noetherian henselian local ring $A$, and let $P\rightarrow M_S$ be a global chart such that the induced map $P\rightarrow M_{S,\bar{s}}/\mathcal{O}_{S,\bar{s}}^\times$ is an isomorphism for the closed point $s$ of $S$. Let $H',H''\in (\text{$p$-div}/S)^{\mr{log}}_{\mr{c}}$ with $H''$ \'etale. Then the functor 
$$\Phi: \mathfrak{Ext}_{S_{\mr{fl}}}(H'',H')\times \mathfrak{Hom}(H''(1),H')\otimes_{\Z}P^{\mr{gp}}\rightarrow \mathfrak{Ext}_{S_{\mr{kfl}}}(H'',H')$$
from (\ref{eq2.8}) is an equivalence of categories.
\end{thm}
\begin{proof}
This follows from the equivalences on the finite levels, see Theorem \ref{thm2.1}.
\end{proof}

Let $S$ be as in Theorem \ref{thm2.3}, and we further assume that the residue characteristic $p$ is positive. Assume that $H\in (\text{$p$-div}/S)^{\mr{log}}_{\mr{d}}$. Similarly to \S 2.2, there is an exact connected-\'etale sequence 
\[0\to H^{\circ}\to H\to H^{\et}\to 0\]
where $H^{\circ}$ (resp. $H^{\et}$) is a connected (resp. \'etale) classical $p$-divisible group. This one gets by observing that the connected-\'etale sequences for the finite levels are compatible with the transition maps. Using this and Corollary \ref{cor2.1}, one gets a complete classification of log-$p$-divisible groups in $(\text{$p$-div}/S)^{\mr{log}}_{\mr{d}}$.
\begin{cor}[Kato]\label{cor2.3}
Let $S$ be an fs log scheme whose underlying scheme is $\Spec A$ for a noetherian henselian local ring $A$ with residue characteristic $p>0$, and let $P\rightarrow M_S$ be a global chart such that the induced map $P\rightarrow M_{S,\bar{s}}/\mathcal{O}_{S,\bar{s}}$ is an isomorphism for the closed point $s$ of $S$. Then any $H\in (\text{$p$-div}/S)^{\mr{log}}_{\mr{d}}$ is constructed out of a unique object in 
$$\mathfrak{Ext}_{S_{\mr{fl}}}(H^{\mr{\acute{e}t}}, H^{\circ})\times \mathfrak{Hom}(H^{\mr{\acute{e}t}}(1),H^{\circ})\otimes_{\Z}P^{\mr{gp}}$$
from the equivalence of categories (\ref{eq2.8}).
\end{cor}
\begin{rmk}\label{rmk2.2}
Let $S$ be a general fs log scheme whose underlying scheme is locally noetherian, and let $H=(H_n)_n$ be extension of an \'etale classical $p$-divisible group $H''$ by another classical $p$-divisible group $H'$. For each $n$, we can make $H_n$ standard \'etale locally on $S$ by Theorem \ref{thm2.2}. With some more effort, we can make all $H_n$ standard at the same time over some \'etale neighborhood of any point $s\in S$. However, it is not clear to the authors how to construct a compatible system of monodromies at all levels.
\end{rmk}

We can however use the theorem to prove that log $p$-divisible groups are formally log smooth.

\begin{thm}\label{thm2.4}
Let $S$ be a locally noetherian fs log scheme on which $p$ is locally nilpotent, and $H$ a log $p$-divisible group over $S$, which lies in $(\text{$p$-div}/S)^{\mr{log}}_{\mr{d}}$. Then $H$ is formally log smooth, i.e. for any strict closed square-zero thickening $T_0\hookrightarrow T$ in $(\mr{fs}/S)$, any element of $H(T_0)$ can be lifted to an element of $H(T)$ \'etale locally on $T$.
\end{thm}

\begin{proof}
We may assume that $T$ is affine. Now we work locally around a geometric point $\bar{t}$ of $T$. Let $\overline{T}$ be the spectrum of the \'etale stalk of the structure sheaf $\mathcal{O}_{T, \bar{t}}$, which is a strict henselian local ring. Then  $\overline{T}=\varprojlim_{(V,\bar{t})} V$ is the limit along the \'etale neighbourhoods of $\bar{t}$. Now $H=\varinjlim H_{n}$, where $H_{n}$ are finite kfl log group schemes. So the $H_{n}$ are in particular of finite presentation and therefore limit preserving functors. Here on the \'etale neighbourhoods we put the log structure that is induced from the log structure on $T$. Hence we get 
\begin{align*}
H(\overline{T}) & = \varinjlim_{n}H_{n}(\overline{T})\\ & = \varinjlim_{n}H_{n}(\varprojlim V)\\  &= \varinjlim_{n}\varinjlim_{V}H_{n}(V)\\ & = \varinjlim_{V}H(V).
\end{align*}
Here the first equality is due to quasi-compactness of $\overline{T}$ and the last equality follows formally since colimits commute with colimits. In particular any section of $H(\overline{T})$ extends to an \'etale neighbourhood.

From this we see that we can assume that $T$ is strict henselian local. Further by replacing $H$ by $H\times_ST$, we may work with the base $T$. It suffices to prove the surjectivity of the map $H(T)\to H(T_0)$. 

Fix a neat chart $P$ of $\mathcal{M}_{T}$, such that $P\cong M_{T,\bar{t}}/\mathcal{O}_{T,\bar{t}}^\times=M_{T_{0},\bar{t}}/\mathcal{O}_{T_{0},\bar{t}}^\times$, where now $\bar{t}$ denotes the closed point of $T$. Now let $0\to H^\circ\to H\to H^{\mr{\acute{e}t}}\to 0$ be the connected-\'etale sequence of $H$, and we have that $H^\circ$ and $H^{\mr{\acute{e}t}}$ are classical. By Corollary \ref{cor2.3}, $H$ admits a decomposition $H_{T}=H^{\mr{cl}}+_{\mr{Baer}} H^{N}$ into the Baer sum of a classical $p$-divisible group $H^{\mr{cl}}\in \mathfrak{Ext}_{S_{\mr{fl}}}(H^{\mr{\acute{e}t}}, H^{\circ})$ and a log $p$-divisible group $H^{N}$ constructed out of the monodromy $N\in \mathfrak{Hom}(H^{\mr{\acute{e}t}}(1), H^{\circ})\otimes_{\Z}P^{\mr{gp}}$ of $H$.
We have the following canonical commutative diagram
$$\xymatrix{
0\ar[r] &H^{\circ}(T)\ar[r]\ar@{->>}[d]^{\alpha} &H(T)\ar[r]\ar[d]^{\beta} &H^{\mr{\acute{e}t}}(T)\ar[r]^-{\delta}\ar[d]_{\cong}^{\gamma} &H^1_{\mr{kfl}}(T, H^{\circ})\ar[d]^{\lambda}     \\
0\ar[r] &H^{\circ}(T_0)\ar[r] &H(T_0)\ar[r] &H^{\mr{\acute{e}t}}(T_0)\ar[r]^-{\delta_0} &H^1_{\mr{kfl}}(T_0,H^{\circ}) 
}$$
with exact rows. Since both $H^{\circ}$ and $H^{\mr{\acute{e}t}}$ are classically formally smooth by \cite[Chap. II, 3.3.13]{mes1}, both $\alpha$ and $\gamma$ are surjective. Moreover, since $H^{\mr{\acute{e}t}}$ is \'etale, $\gamma$ is even an isomorphism. To show the surjectivity of $\beta$, it suffices to show that the restriction $\lambda\vert_{\mr{im}(\delta)}$ of $\lambda$ to $\mr{im}(\delta)$ is injective by the five lemma. 

Let $\varepsilon:(\mr{fs}/T)_{\mr{kfl}}\rightarrow (\mr{fs}/T)_{\mr{fl}}$ be the canonical forgetful map between the two sites. By Proposition \ref{propD} the choice of $P$ also provides a splitting 
\begin{equation}\label{eq2.10}
H^1_{\mr{kfl}}(T, H^{\circ})\cong H^{1}_{\mr{fl}}(T, H^{\circ})\oplus \mr{Hom}_T(\mathbb{Z}_{p}(1), H^{\circ})\otimes_{\Z}P^{\mr{gp}}.
\end{equation}

For all $n\geq 1$, the sheaf $\mc{H}om_T(\Z/p^n\Z(1),H^{\circ}_n)$ is representable by an \'etale group scheme over $T$ by Lemma \ref{lemA.1}. Thus we have
\[\Gamma(T,\mc{H}om_T(\Z/p^n\Z(1),H^{\circ}_n))=\Gamma(T_0,\mc{H}om_T(\Z/p^n\Z(1),H^{\circ}_n))\]
Therefore the map $\lambda$ is injective on $\mr{Hom}_T(\mathbb{Z}_{p}(1), H^{\circ})\otimes_{\Z}P^{\mr{gp}}$.

Recall that we have $H=H^{\mr{cl}}+_{\mr{Baer}}H^N$. Let $\mr{pr}_{\mr{cl}}$ (resp. $\mr{pr}_N$) be the projection from $H^{\mr{cl}}$ (resp. $H^N$) to $H^{\mr{\acute{e}t}}$. For any $x\in H^{\mr{\acute{e}t}}(T)$, $\delta(x)$ can be represented by the sum of the two torsors $\mr{pr}_{\mr{cl}}^{-1}(x)\in H^1_{\mr{fl}}(T,H^{\circ})$ and $\mr{pr}_N^{-1}(x)\in H^1_{\mr{kfl}}(T,H^{\circ})$, and this sum is the direct decomposition with respect to (\ref{eq2.10}).
Let $x$ be such that $\delta(x)\in \mr{ker}(\lambda)$. Since $\lambda$ is injective on $\mr{Hom}_T(\mathbb{Z}_{p}(1), H^{\circ})\otimes_{\Z}P^{\mr{gp}}$, it follows that $\delta(x)$ is represented by its classical part $\mr{pr}_{\mr{cl}}^{-1}(x)$, and $0=\lambda(\delta(x))=\lambda(\mr{pr}_{\mr{cl}}^{-1}(x))$. 
Since $H^{\mr{cl}}$ is formally smooth by \cite[Chap. II, 3.3.13]{mes1}, the homomorphism $\beta^{\mr{cl}}$ in the diagram 
$$\xymatrix{
0\ar[r] &H^{\circ}(T)\ar[r]\ar@{->>}[d]^{\alpha} &H^{\mr{cl}}(T)\ar[r]\ar@{->>}[d]^{\beta^{\mr{cl}}} &H^{\mr{\acute{e}t}}(T)\ar[r]^-{\delta^{\mr{cl}}}\ar[d]_{\cong}^{\gamma} &H^1_{\mr{fl}}(T,H^{\circ})\ar[d]^{\lambda_1}     \\
0\ar[r] &H^{\circ}(T_0)\ar[r] &H^{\mr{cl}}(T_0)\ar[r] &H^{\mr{\acute{e}t}}(T_0)\ar[r]^-{\delta_0^{\mr{cl}}} &H^1_{\mr{fl}}(T_0,H^{\circ}) 
}$$ 
is surjective. An easy diagram chasing tells us that $\delta(x)=\delta^{\mr{cl}}(x)$ is trivial (here the bijectivity of $\gamma$ is used). Hence $\lambda\vert_{\mr{im}(\delta)}$ is injective. This finishes the proof.
\end{proof}

\section{Log $p$-divisible groups associated to log 1-motives}
In this section we want to study the finite Kummer flat group log schemes associated to log $1$-motives. In this case, we will see that the splittings (in the sense of the previous section) are induced from the monodromy pairings of log 1-motives (see Definition \ref{defn3.1} below).

Let $S$ be an fs log scheme with its underlying scheme locally noetherian. Let $G$ be an extension of an abelian scheme by a torus $T$, and let $G_{\mr{log}}$ be the pushout of $G$ in the category of sheaves of abelian groups on $(\mr{fs}/S)_{\mr{kfl}}$ along the log-enlargement $T\to T_{\mr{log}}$. Here $T_{\mr{log}}=\mathcal{H}om(X, \Gml)$ with $X$ the character sheaf of the torus $T$. Then a log 1-motive $\mathbf{M}=[Y\xrightarrow{u} G_{\mr{log}}]$ over $S$, is a morphism of commutative group sheaves, where $Y$ is a lattice, i.e. an \'etale locally finite free abelian group sheaf. See \cite[Def. 2.2]{k-k-n2} for the details about log 1-motives. Note that $G_{\mr{log}}$ in \cite{k-k-n2} is defined on the site $(\mr{fs}/S)_{\rm{\acute{e}t}}$, and agrees with ours by \cite[Prop. 2.1]{zha3}.

The composition
\begin{center}$Y\xrightarrow{u} G_{\mr{log}}\to G_{\mr{log}}/G=T_{\mr{log}}/T=\mathcal{H}om(X, \Gml/\Gm)$
\end{center}
corresponds to a pairing 
\[\langle-,-\rangle_\mathbf{M}: X\times Y \to \Gml/\Gm.\]
\begin{defn}\label{defn3.1}
We call the pairing $\langle-,-\rangle_\mathbf{M}$ the \textbf{monodromy pairing} of the log 1-motive $\mathbf{M}$.
\end{defn}

\begin{prop}\label{prop3.1}
Suppose $G$ is an extension of an abelian scheme $B$ by a torus $T$ over the underlying scheme of $S$. Then \'etale locally on $S$ we have a decomposition $u=u_1+u_2$, where $u_i:Y\rightarrow G_{\mr{log}}$ for $i=1,2$, such that $u_1$ factorizes as $Y\rightarrow G\hookrightarrow G_{\mr{log}}$ and $u_2$ factorizes as $Y\rightarrow T_{\mr{log}}\hookrightarrow G_{\mr{log}}$.
\end{prop}
\begin{proof}
Let $\bar{u}$ be the composition 
\[Y\xrightarrow{u}G_{\mr{log}}\rightarrow (G_{\mr{log}}/G)_{S_{\mr{\acute{e}t}}}=(T_{\mr{log}}/T)_{S_{\mr{\acute{e}t}}}.\]
The short exact sequence $0\rightarrow T\rightarrow T_{\mr{log}}\rightarrow (T_{\mr{log}}/T)_{S_{\mr{\acute{e}t}}}\rightarrow 0$ gives rise to an exact sequence
$$\rightarrow \mr{Hom}_{S}(Y,T_{\mr{log}})\rightarrow\mr{Hom}_{S}(Y,T_{\mr{log}}/T)\xrightarrow{\delta}\mr{Ext}^1_{S_{\mr{\acute{e}t}}}(Y,T)\rightarrow.$$
\'Etale locally, $\mr{Ext}^1_{S_{\mr{\acute{e}t}}}(Y,T)=\mr{Ext}^1_{S_{\mr{\acute{e}t}}}(\Z^r,\Gm^t)=\mr{Ext}^1_{S_{\mr{\acute{e}t}}}(\Z,\Gm)^{rt}\cong\mr{Pic}(S)^{rt}$. Since a line bundle is Zariski locally trivial, $\delta(\bar{u})$ is zero \'etale locally on $S$. It follows that there exists $u_2\in\mr{Hom}_S(Y,T_{\mr{log}})$ lifting $\bar{u}$ \'etale locally on $S$. Let $u_1:=u-u_2$ as a homomorphism from $Y$ to $G_{\mr{log}}$, it is easy to see that $u_1$ factors through $G$. This finishes the proof.
\end{proof}

\begin{prop}\label{prop3.2}
Assume that $S$ admits a global chart $\alpha:P\rightarrow M_S$, where $P$ is an fs monoid.
Any bilinear map $\langle-,-\rangle:X\times Y\rightarrow P^{\mr{gp}}_S$ which lifts the monodromy pairing $\langle-,-\rangle_\mathbf{M}$ of $\mathbf{M}$, gives rise to a splitting $u=u_{1}+u_{2}$, where $u_2:Y\rightarrow T_{\mr{log}}\hookrightarrow G_{\mr{log}}$ is the map induced by $\langle-,-\rangle$ and $u_1=u-u_2:Y\rightarrow G_{\mr{log}}$ factors through $G\hookrightarrow G_{\mr{log}}$.

And such a bilinear map always exists \'etale locally.
\end{prop}
\begin{proof}
The chart $\alpha:P\rightarrow M_S$ gives rise to a push-out diagram
$$\xymatrix{
\alpha^{-1}(\mc{O}_S^{\times})\ar[r]\ar[d] &\mc{O}_S^{\times}\ar[d]  \\
P\ar[r] &M_S
}.$$
Since taking group envelope commutes with colimits, we get another push-out digram
$$\xymatrix{
\alpha^{\mr{gp},-1}(\mc{O}_S^{\times})\ar[r]\ar[d] &\mc{O}_S^{\times}\ar[d]  \\
P^{\mr{gp}}_S\ar[r] &M^{\mr{gp}}_S
},$$
where $\alpha^{gp}:P^{\mr{gp}}_S\rightarrow M^{\mr{gp}}_S$ denotes the homomorphism induced from $\alpha$. This push-out diagram gives a short exact sequence
$$0\rightarrow \alpha^{\mr{gp},-1}(\mc{O}_S^{\times})\rightarrow P^{\mr{gp}}_S\rightarrow \overline{M}^{\mr{gp}}_S\rightarrow 0$$
on the small \'etale site of $S$ with $\overline{M}^{\mr{gp}}_S:=M^{\mr{gp}}_S/\mc{O}_S^{\times}$. Applying $\mr{Hom}_{S_{\mr{\acute{e}t}}}(Y\otimes_{\Z}X,-)$ to this short exact sequence, we get an exact sequence
$$\mr{Hom}_{S}(Y\otimes_{\Z}X,P^{\mr{gp}}_S)\rightarrow \mr{Hom}_{S}(Y\otimes_{\Z}X,\overline{M}^{\mr{gp}}_S)\xrightarrow{\delta} \mr{Ext}_{S_{\mr{\acute{e}t}}}^1(Y\otimes_{\Z}X,\alpha^{\mr{gp},-1}(\mc{O}_S^{\times})).$$
Let $\bar{\gamma}:Y\otimes_{\Z}X\rightarrow(\Gml/\Gm)_{S_{\mr{\acute{e}t}}}$ be the homomorphism induced by the canonical pairing. Since both $X$ and $Y$ are \'etale locally constant, the map $\bar{\gamma}$ is determined by its induced map $Y\otimes_{\Z}X\rightarrow\overline{M}^{\mr{gp}}_S$ which we still denote by $\bar{\gamma}$ by abuse of notation. Since 
$$\mr{Ext}_{S_{\mr{\acute{e}t}}}^1(Y\otimes_{\Z}X,\alpha^{\mr{gp},-1}(\mc{O}_S^{\times}))=H^1(S_{\mr{\acute{e}t}},\mc{H}om_{S}(Y\otimes_{\Z}X,\alpha^{\mr{gp},-1}(\mc{O}_S^{\times}))),$$
the element $\delta(\bar{\gamma})$ is \'etale locally trivial on $S$. Therefore, \'etale locally on $S$ the map $\bar{\gamma}$ lifts to a homomorphism $\gamma:Y\otimes_{\Z}X\rightarrow P^{\mr{gp}}_S$.  The map $\gamma$ gives rise to a homomorphism $Y\otimes_{\Z}X\rightarrow\Gml$ which corresponds to a homomorphism $u_2:Y\rightarrow T_{\mr{log}}$. Obviously, $u_2$ lifts $\bar{u}$. Let $u_1:=u-u_2$, we have that $u_1$ factors through $G\hookrightarrow G_{\mr{log}}$. This finishes the proof.
\end{proof}

\begin{defn}
For any positive integer $n$, let 
\[T_n(\mathbf{M}):=H^{-1}(\mathbf{M}\otimes^L\Z/n\Z).\]
\end{defn}

\begin{prop}\label{prop3.3}
Let $S$ be a locally noetherian fs log scheme, $M=[Y\xrightarrow{u}G_{\mr{log}}]$
a log 1-motive over $S$, and $n$ a positive integer. Then we have the following.
\begin{enumerate}[(1)]
\item $T_n(M)$ fits into the following exact sequence
$$0\rightarrow G[n]\rightarrow T_n(\mathbf{M})\rightarrow Y/nY\rightarrow0$$
of sheaves of abelian groups on $(\mr{fs}/S)_{\mr{kfl}}$.
\item $T_n(\mathbf{M})\in(\mr{fin}/S)_{\mr{d}}$.
\item Let $m$ be another positive integer. Then the map $T_{mn}(\mathbf{M})\rightarrow T_{n}(\mathbf{M})$ induced by $\Z/mn\Z\xrightarrow{m}\Z/n\Z$ is surjective.
\end{enumerate}
\end{prop}
\begin{proof}
The map $Y\xrightarrow{n}Y$ is clearly injective. Consider the following commutative diagram
$$\xymatrix{
0\ar[r] &G\ar[r]\ar[d]^n &G_{\mr{log}}\ar[r]\ar[d]^n &\mc{H}om_{S_{\mr{kfl}}}(X,\Gml/\Gm)\ar[r]\ar[d]^n &0   \\
0\ar[r] &G\ar[r] &G_{\mr{log}}\ar[r] &\mc{H}om_{S_{\mr{kfl}}}(X,\Gml/\Gm)\ar[r] &0
}$$
with exact rows. By \cite[Prop. 4.2]{kat2} we also have the following commutative diagram with exact rows
$$\xymatrix{
0\ar[r] &\Z/n\Z(1)\ar[r]\ar@{=}[d] &\Gm\ar[r]^n\ar[d] &\Gm\ar[r]\ar[d] &0   \\
0\ar[r] &\Z/n\Z(1)\ar[r] &\Gml\ar[r]^n &\Gml\ar[r] &0
}$$

From the last diagram one sees that 
\[\mc{H}om_{S_{\mr{kfl}}}(X,\Gml/\Gm)\xrightarrow{n}\mc{H}om_{S_{\mr{kfl}}}(X,\Gml/\Gm)\]
is an isomorphism. Moreover, $G\xrightarrow{n}G$ is surjective, so we get that $G[n]\cong G_{\mr{log}}[n]$ and that $G_{\mr{log}}\xrightarrow{n}G_{\mr{log}}$ is surjective. By \cite[\S 3.1]{ray2}, we get a short exact sequence 
\[0\rightarrow G[n]\rightarrow T_n(\mathbf{M})\rightarrow Y/nY\rightarrow0.\]
This finishes the proof of part (1).

For part (2), note that by \cite[Prop. 2.3]{kat4} and the short exact sequence from part (1), we have that $T_n(\mathbf{M})\in(\mr{fin}/S)_{\mr{r}}$. Let $\mathbf{M}^{\vee}$ be the dual log 1-motive of $\mathbf{M}$ (see \cite[\S 2.7]{k-k-n2}). Then we can identify the Cartier dual of $T_n(\mathbf{M})$ with $T_n(\mathbf{M}^{\vee})$. Since we also have $T_n(\mathbf{M}^{\vee})\in (\mr{fin}/S)_{\mr{r}}$, we get $T_n(\mathbf{M})\in(\mr{fin}/S)_{\mr{d}}$.

Part (3) is clearly true for the two log 1-motives $[Y\rightarrow 0]$ and $[0\rightarrow G_{\mr{log}}]$. It follows that it also holds for $\mathbf{M}$.
\end{proof}

\begin{defn}
By Proposition \ref{prop3.3}, 
\[\mathbf{M}[p^{\infty}]:=\varinjlim_nT_{p^n}(\mathbf{M})\]
is an object of $(\text{$p$-div}/S)^{\mr{log}}_{\mr{d}}$, we call it the \textbf{log $p$-divisible group associated to $\mathbf{M}$}.
\end{defn}

\begin{prop}\label{prop3.4}
Assume that $S$ admits a global chart $\alpha:P\rightarrow M_S$, where $P$ is an fs monoid. Let $\mathbf{M}=[Y\xrightarrow{u} G_{\mr{log}}]$ be a log 1-motive over $S$, such that $X$ and $Y$ are constant. Let $n$ be a positive integer.

Assume further that  there exists a bilinear map as in Proposition \ref{prop3.2} 
\[\langle-,-\rangle:X\times Y\rightarrow P^{\mr{gp}}_S,\]
such that the homomorphism $u_2:Y\xrightarrow{\beta} T_{\mr{log}}\hookrightarrow G_{\mr{log}}$,  which is induced from the pairing $\langle-,-\rangle$, gives rise to a decomposition $u=u_1+u_2$ with $u_1:Y\xrightarrow{\alpha} G\hookrightarrow G_{\mr{log}}$. Then we have the following.
\begin{enumerate}[(1)]
\item $T_n(\mathbf{M})$ is a standard extension of $Y/nY$ by $G[n]$ with respect to the given chart.
\item Let $f_1,\cdots,f_m$ (resp. $e_1,\cdots,e_r$) be a basis of $X$ (resp. $Y$), and let $f_1^*,\cdots,f_m^*$ (resp. $e_1^*,\cdots,e_r^*$) be the corresponding dual basis of $X^*=\mr{Hom}_{\Z}(X,\Z)$ (resp. $Y^*=\mr{Hom}_{\Z}(Y,\Z)$). The pairing $\langle, \rangle:X\times Y\rightarrow P^{\mr{gp}}_S$ gives rise to a canonical pseudo-monodromy of $T_n(M)$, namely the element $\sum_{i,j}\delta_{ij}\otimes \langle f_i,e_j\rangle$ in the group
\[\mr{Hom}_S(Y/nY(1),T[n])\otimes P^{\mr{gp}}=\mr{Hom}_S(X/nX,Y^*/nY^*)\otimes  P^{\mr{gp}}.\]
Here $\delta_{ij}$ denotes the map $\bar{f}_k\mapsto \begin{cases}0, &\text{if $k\neq i$} \\  \bar{e}_j^*, &\text{if $k=i$.} \end{cases}$
\item For varying $n$, the pseudo-monodromies from (2) are compatible with each other. 
\end{enumerate}
\end{prop}
\begin{proof}
Let $\mathbf{M}_1:=[Y\xrightarrow{u_1} G_{\mr{log}}]$, $\mathbf{M}_{\alpha}:=[Y\xrightarrow{\alpha} G]$, $M_2=[Y\xrightarrow{u_2} G_{\mr{log}}]$, and $\mathbf{M}_{\beta}:=[Y\xrightarrow{\beta} T_{\mr{log}}]$. Then $T_n(\mathbf{M}_1)$ is canonically isomorphic to $T_n(\mathbf{M}_{\alpha})$, $T_n(M_2)$ is given by the push-out of $T_n(M_{\beta})$ along the canonical map $T[n]\hookrightarrow G[n]$, and $T_n(\mathbf{M})=T_n(\mathbf{M}_1)+_{\mr{Baer}}T_n(\mathbf{M}_2)$. Since $T_n(\mathbf{M}_1)=T_n(\mathbf{M}_{\alpha})\in(\mr{fin}/S)_{\mr{c}}$, we are reduced to show that $T_n(\mathbf{M}_2)$ is a standard extension of $Y/nY$ by $G[n]$.

We first deal with the case that $G=T$. Applying the functor $\varepsilon_*$ to the short exact sequence $0\rightarrow T[n]\rightarrow T_n(\mathbf{M}_2)\rightarrow Y/nY\rightarrow 0$ on $(\mr{fs}/S)_{\mr{kfl}}$, we get a long exact sequence
$$0\rightarrow T[n]\rightarrow T_n(\mathbf{M}_2)\rightarrow Y/nY\xrightarrow{\delta} R^1\varepsilon_*T[n]$$
on $(\mr{fs}/S)_{\mr{fl}}$.
For any $1\leq j\leq r$, $\delta(\bar{e}_j)$ is represented by the $T[n]$-torsor 
$$\{a\in T_{\mr{log}}\mid a^n=u_2(e_j)\}.$$
Let $\bar{u}$ be the composition $Y\xrightarrow{u_2}T_{\mr{log}}\rightarrow T_{\mr{log}}/T\cong\mc{H}om_{S_{\mr{kfl}}}(X,\Gml/\Gm)$. Then $\bar{u}(e_j)(f_i)=p_{ij}$, where $p_{ij}\in P^{\mr{gp}}$ is such that $\langle f_i,e_j\rangle=p_{ij}$ for each $1\leq i\leq m$. Therefore, under the identification 
\begin{align*}
R^1\varepsilon_*T[n]&\cong\mc{H}om_{S_{\mr{fl}}}(\Z/n\Z(1),T[n])\otimes (\Gml/\Gm)_{S_{\mr{fl}}} \\&\cong X^*/nX^*\otimes (\Gml/\Gm)_{S_{\mr{fl}}},
\end{align*}
we have $\delta(\bar{e}_j)=\sum_{i=1}^m\bar{f}_i^*\otimes p_{ij}$. It follows that $\delta$ can be lifted canonically to a homomorphism 
$$\bar{\delta}:Y/nY\rightarrow \mc{H}om_{S_{\mr{fl}}}(\Z/n\Z(1),T[n])\otimes P^{\mr{gp}}=X^*/nX^*\otimes P^{\mr{gp}},\bar{e}_j\mapsto\sum_{i=1}^m\bar{f}_i\otimes p_{ij}.$$ Therefore $T_n(\mathbf{M}_2)$ is standard with respect to the given chart by Proposition \ref{prop2.1}.

In general, applying the functor $\varepsilon_*$ to the push-out diagram
$$\xymatrix{
0\ar[r] &T[n]\ar[r]\ar[d] &T_n(\mathbf{M}_{\beta})\ar[r]\ar[d] &Y/nY\ar[r]\ar@{=}[d] &0  \\
0\ar[r] &G[n]\ar[r] &T_n(\mathbf{M}_2)\ar[r] &Y/nY\ar[r] &0
},$$
we get the following commutative diagram
$$\xymatrix{
0\ar[r] &T[n]\ar[r]\ar[d] &T_n(\mathbf{M}_{\beta})\ar[r]\ar[d] &Y/nY\ar[r]\ar@{=}[d] &R^1\varepsilon_*T[n]\ar[d]  \\
0\ar[r] &G[n]\ar[r] &T_n(\mathbf{M}_2)\ar[r] &Y/nY\ar[r] &R^1\varepsilon_*G[n]
}$$
with exact rows. Therefore we are reduced to the case $G=T$ by Proposition \ref{prop2.1}. This finishes the proof of part (1). The searched canonical pseudo-monodromy in part (2) is given by the canonical lift $\bar{\delta}$ of the map $\delta$. The compatibility in part (3) is clear from the construction of the canonical pseudo-monodromy.
\end{proof}
\begin{rmk}\label{rmk3.1}
If we have a decomposition of $\mathbf{M}$ into $u_{1}$ and $u_{2}$ as above, it is easy to see that for any $n\geq 1$ we also get a decomposition $T_{n}(\mathbf{M})=T_{n}(u_{1})+_{\mr{Baer}} T_{n}(u_{2})$, where $T_{n}(u_{1})$ is classical. The proposition shows that this decomposition is the same as the one obtained from the proof of proposition \ref{prop2.1}, if one uses the lift 
\[\bar{\delta}_{T_{n}(\mathbf{M})}:Y/nY\xrightarrow{\bar{\delta}} \mc{H}om_{S_{\mr{fl}}}(\Z/n\Z(1),T[n])\otimes P^{\mr{gp}}\to \mc{H}om_{S_{\mr{fl}}}(\Z/n\Z(1),G[n])\otimes P^{\mr{gp}}\]
of $Y/nY\to R^1\varepsilon_*G[n]$.

Assume that the underlying scheme of $S$ is the spectrum of a henselian local ring $R$ and $P\to M_S$ is a global chart which is neat at the closed point $s$ of $S$. By the same argument as in Lemma \ref{lem2.2}, we have 
\[\mr{Hom}_S(X\otimes_{\Z}Y,P^{\mr{gp}})\xrightarrow{\cong}\mr{Hom}_S(X\otimes_{\Z}Y,(\Gml/\Gm)_{S_{\mr{fl}}}),\]
therefore the lifting of the canonical pairing is uniquely determined by the choice of the chart, and further the pure log part $u_2:Y\to T_{\mr{log}}\hookrightarrow G_{\mr{log}}$ is uniquely determined by the choice of the chart. When $R$ is further a complete discrete valuation ring, choosing a uniformizer determines a neat chart, and therefore determines a decomposition $u=u_1+u_2$. Such a determination is compatible with Raynaud's decomposition of strict 1-motives over $\mr{Frac}(R)$ with integral geometric monodromy, see \cite[Cor. 4.5.1]{ray2} and \cite[\S 4.2]{zha5}. In the case of a DVR endowed with its canonical log structure, the compatibility with the work of Kato has also already been treated in \cite{b-c-c1}. 
\end{rmk}

\section{The Serre-Tate theorem for log abelian varieties with constant degeneration}
In this section we want to show how one can adapt Drinfeld's proof of the classical Serre-Tate theorem to the situation of log abelian varieties with constant degeneration using the results from the previous sections.

\subsection{Log abelian varieties with constant degeneration}
We first recall the notion of a log abelian variety with constant degeneration from \cite[\S 3]{k-k-n2}. Let $S$ be an fs log scheme. Let $\mathbf{M}=[Y\xrightarrow{u} G_{\mr{log}}]$ be a log 1-motive over $S$ with its dual log 1-motive $\mathbf{M}^*=[X\xrightarrow{u^*} G^{*}_{\mr{log}} ]$. Consider the subsheaf 
\[\mathcal{H}om(X, \Gml/\Gm)^{(Y)}\]
of $\mathcal{H}om(X, \Gml/\Gm)$ on the site $(\mr{fs}/S)_{\mr{\acute{e}t}}$, whose sections over $V\in (\mr{fs}/S)$ are defined to be:
\begin{center}
$\{ \phi\in \mathcal{H}om(X, \Gml/\Gm)(V) :$ for any $v\in V$ and $x\in X_{\bar{v}}$, there exist $y_{v, x}, y_{v, x}'\in Y_{\bar{v}}$, such that $\langle x,y_{v, x}\rangle_{\bar{v}} \vert \phi(x)_{\bar{v}} \vert \langle x,y'_{v, x}\rangle_{\bar{v}} \}$.
\end{center}
Here $\langle - , - \rangle:X\times Y\to \Gml/\Gm$ denotes the monodromy pairing of $M$. We then define $G_{\mr{log}}^{(Y)}$ as the preimage of $\mathcal{H}om(X, \Gml/\Gm)^{(Y)}$ under the natural map 
\begin{center}
$G_{\mr{log}}\to G_{\mr{log}}/G\cong \mathcal{H}om(X, \Gml/\Gm)$.
\end{center}
A log abelian variety with constant degeneration over $S$ is defined to be a sheaf of abelian groups on $(\mr{fs}/S)_{\mr{\acute{e}t}}$, which is isomorphic to $G_{\mr{log}}^{(Y)}/Y$ for a pointwise polarizable log 1-motive $[Y\xrightarrow{u} G_{\mr{log}}]$. For the notion of being pointwise polarizable for a log 1-motive we refer to \cite[Definition 2.8]{k-k-n2}. One of the main results of \cite{k-k-n2} is then the following:
\begin{thm}\cite[Theorem 3.4]{k-k-n2}
The association 
\begin{center}
$[Y\xrightarrow{u} G_{\mr{log}}] \mapsto G^{(Y)}/Y$
\end{center}
defines an equivalence of categories from pointwise polarizable log 1-motives over $S$ to log abelian varieties with constant degeneration over $S$.
\end{thm}

\begin{rmk}\label{rmk5.1}
Pointwise polarizable log $1$-motives satisfy effective \'etale descent. For this assume that $\tilde{S}\to S$ is \'etale and $[\tilde{u}:\tilde{Y}\to \tilde{G}_{\mr{log}}]$ is a pointwise polarizable log $1$-motive on $\tilde{S}$, with a gluing datum on $\tilde{S}\times_{S}\tilde{S}$. Then by \cite[Proposition 2.5]{k-k-n2} one also gets a descent datum on the group scheme $\tilde{G}$, which is an extension of an abelian scheme by a torus. Hence $\tilde{G}$ descends to such an extension $G$ over $S$. Moreover, $\tilde{Y}$ also descends to some $Y$ over $S$ and then $\tilde{u}$ glues to a morphism $[u:Y\to G_{\mr{log}}]$. By the definition of being pointwise polarizable, $u$ is still pointwise polarizable.

By this one then also gets effective \'etale descent for log abelian varieties with constant degeneration.
\end{rmk} 
First we remark that the log $p$-divisible group of a log abelian variety with constant degeneration coincides with the one coming from the log 1-motive.
\begin{lem}\label{cor3.1}
Let $A$ be a weak log abelian variety with constant degeneration over $S$, see \cite[\S 1.7]{k-k-n4}. By \cite[\S 1.7]{k-k-n4}, there exists an admissible and nondegenerate
log 1-motive $\mathbf{M}=[Y\rightarrow G_{\mr{log}}]$ such that $A=G_{\mr{log}}^{(Y)}/Y$. Then we have the following.
\begin{enumerate}[(1)]
\item For any positive integer $n$, $A[n]$ is canonically identified with $T_n(\mathbf{M})$.
\item $A[p^{\infty}]:=\bigcup_{n>0}A[p^n]$ is canonically identified with $\mathbf{M}[p^{\infty}]$.
\end{enumerate}
\end{lem}
\begin{proof}
Consider the following two commutative diagrams
$$\xymatrix{
0\ar[r] &Y\ar[r]\ar[d]^n &G_{\mr{log}}^{(Y)}\ar[r]\ar[d]^n &A\ar[r]\ar[d]^n &0  \\
0\ar[r] &Y\ar[r] &G_{\mr{log}}^{(Y)}\ar[r] &A\ar[r] &0 
}$$
and
$$\xymatrix{
0\ar[r] &Y\ar[r]\ar[d]^n &G_{\mr{log}}\ar[r]\ar[d]^n &G_{\mr{log}}/Y\ar[r]\ar[d]^n &0  \\
0\ar[r] &Y\ar[r] &G_{\mr{log}}\ar[r] &G_{\mr{log}}/Y\ar[r] &0 
}$$
with exact rows. By Lemma \cite[Lem. 3.2]{zha3}, we have that $G_{\mr{log}}^{(Y)}\xrightarrow{n}G_{\mr{log}}^{(Y)}$ is surjective with kernel $G[n]$. We also have that $G_{\mr{log}}\xrightarrow{n}G_{\mr{log}}$ is surjective with kernel $G[n]$. It follows that $A[n]$ and $(G_{\mr{log}}/Y)[n]$ fit into the following commutative diagram 
$$\xymatrix{
0\ar[r] &G[n]\ar[r]\ar@{=}[d] &A[n]\ar[r]\ar[d] &Y/nY\ar[r]\ar@{=}[d] &0  \\
0\ar[r] &G[n]\ar[r] &(G_{\mr{log}}/Y)[n]\ar[r] &Y/nY\ar[r] &0 
}$$
with exact rows. Hence $A[n]$ is canonically identified with $(G_{\mr{log}}/Y)[n]$. The canonical quasi-isomorphism 
\[[Y\rightarrow G_{\mr{log}}]\rightarrow G_{\mr{log}}/Y\]
gives rise to canonical isomorphism $T_n(\mathbf{M})\cong (G_{\mr{log}}/Y)[n]$. It follows that $A[n]$ is canonically identified with $T_n(\mathbf{M})$.
\end{proof}

\begin{lem}\label{lem5.2}
Let $F_1\in (\mr{fin}/S)_{\mr{c}}$ and let $F_2$ be a subobject of $F_1$ in $(\mr{fin}/S)_{\mr{r}}$. Then we have $F_2\in (\mr{fin}/S)_{\mr{c}}$. In other words, the category $(\mr{fin}/S)_{\mr{c}}$ is closed under subobjects in $(\mr{fin}/S)_{\mr{r}}$.
\end{lem}
\begin{proof}
We need to show that the log scheme representing $F_2$ has log structure induced from $S$. This can be verified on each fiber over $S$. Hence we are reduced to the case that $S$ is a log point, which is \cite[Lem. 3.1]{zha3}
\end{proof}

The following lemma generalizes \cite[Lem. 3.2]{zha3}.

\begin{lem}\label{lem5.3}
Let $\mathbf{M}=[Y\rightarrow G_{\mr{log}}],\mathbf{M}'=[Y'\rightarrow G'_{\mr{log}}]$ be two non-degenerate log 1-motives over $S$, $(f_{-1},f_0):\mathbf{M}\rightarrow \mathbf{M}'$ a homomorphism of log 1-motives, and $f_{\mr{c}}:G\rightarrow G'$ the map induced by $f_0$. Let $X$ (resp. $X'$) be the character group of the torus part $T$ (resp. $T'$) of $G$ (resp. $G'$), $\mc{Q}$ (resp. $\mc{Q}'$) the sheaf $\mc{H}om(X,\Gml/\Gm)^{(Y)}$ (resp. $\mc{H}om(X',\Gml/\Gm)^{(Y')}$), $f_{\mr{l}}:X'\rightarrow X$ the map induced by $f_{\mr{c}}$, and $\tilde{f}_{\mr{d}}:\mc{Q}\rightarrow \mc{Q}'$ the map induced by $f_{\mr{l}}$. If $f_{\mr{c}}$ is an isogeny, then the map $\tilde{f}:G_{\mr{log}}^{(Y)}\rightarrow G'^{(Y')}_{\mr{log}}$ induced by $(f_{-1},f_0)$ is surjective with kernel $\mr{Ker}(f_{\mr{c}})$, and the map $\tilde{f}_{\mr{d}}$ is bijective.
\end{lem}
\begin{proof}
Without loss of generality, we may assume that the underlying scheme of $S$ is noetherian. Then the proof of \cite[Lem. 3.2]{zha3} works also here.
\end{proof}

The following proposition generalizes \cite[Prop. 3.2]{zha3}.
\begin{prop}\label{prop5.1}
Let $A$ be a log abelian variety with constant degeneration over $S$, and $F\in (\mr{fin}/S)_{\mr{r}}$ a subgroup sheaf of $A$. Let $n$ be a positive integer which kills $F$, and let $\mathbf{M}=[Y\xrightarrow{u}G_{\mr{log}}]$ be the pointwise polarizable log 1-motive corresponding to $A$. Let $F':=\mr{ker}(F\to A[n]\to Y/nY)$ and $F'':=F/F'$. Assume that $F'\in (\mr{fin}/S)_{\mr{r}}$. Then
\begin{enumerate}[(1)]
\item both $F'$ and $F''$ lie in $(\mr{fin}/S)_{\mr{c}}$;
\item there exists a pointwise polarizable log 1-motive $\mathbf{M}'=[Y'\xrightarrow{u'}G'_{\mr{log}}]$ and a morphism $(f_{-1},f_0):\mathbf{M}\to \mathbf{M}'$ such that $F'=\mr{ker}(G\xrightarrow{f_{\mr{c}}}G')$ and $F''\cong Y'/Y$;
\item the morphism $(f_{-1},f_0)$ induces an isogeny $A\to A'$ of log abelian varieties with constant degeneration with kernel $F$, where $A'$ is the one associated to $\mathbf{M}'$.
\end{enumerate}
\end{prop}
\begin{proof}
The proof is very similar to the proof of \cite[Prop. 3.2]{zha3}. By construction, $F'$ and $F''$ fit into the following commutative diagram
\[\xymatrix{
0\ar[r] &F'\ar[r]\ar@{^(->}[d] &F\ar[r]\ar@{^(->}[d] &F''\ar[r]\ar@{^(->}[d] &0 \\
0\ar[r] &G[n]\ar[r] &A[n]\ar[r] &Y/nY\ar[r] &0
}.\]
Then part (1) follows from Lemma \ref{lem5.2}, using that $G[n]$ and $Y/nY$ are classical.

For part (2) let $E$ be the pullback of $G_{\mr{log}}^{(Y)}$ along $F\subset A$
$$\xymatrix{
0\ar[r] &Y\ar[r]\ar@{=}[d] &E\ar[r]\ar@{^{(}->}[d] &F\ar[r]\ar@{^{(}->}[d] &0 \\
0\ar[r] &Y\ar[r]  &G_{\mr{log}}^{(Y)}\ar[r]  &A \ar[r] &0,
}$$
and let $E_{\mr{tor}}$ be the torsion subsheaf of $E$ and $Y'=E/E_{\mr{tor}}$. Now, $$G_{\mr{log}}^{(Y)}/G=\mc{H}om(X,\Gml/\Gm)^{(Y)}$$
is torsion-free and $E_{\mr{tor}}$ maps into $G$. Therefore, we have $E_{\mr{tor}}=F'$ and $Y'/Y=F''\subset Y/nY$, and $Y'$ is \'etale locally constant. Let $G':=G/E_{\mr{tor}}=G/F'$. Then the inclusion $E\hookrightarrow G_{\mr{log}}^{(Y)}\hookrightarrow G_{\mr{log}}$ yields a homomorphism $Y'\rightarrow G_{\mr{log}}/F'=G'_{\mr{log}}$ by taking the quotients by $F'$. In this way, we get a log 1-motive $\mathbf{M}':=[Y'\xrightarrow{u'} G'_{\mr{log}}]$ together with a homomorphism 
\[(f_{-1},f_0):\mathbf{M}:=[Y\xrightarrow{u} G_{\mr{log}}]\rightarrow \mathbf{M}'.\] 

To show that $\mathbf{M}'$ is pointwise polarizable, we may assume that $S$ is a log point with its underlying field separably closed. The existence of a polarization in this case is shown in the proof of \cite[Prop. 3.2]{zha3}. Let now $A'$ be the log abelian variety with constant degeneration associated to $\mathbf{M}'$. Then by \cite[Thm. 3.4]{k-k-n2} the homomorphism $(f_{-1},f_0)$ gives rise to a homomorphism $f:A\rightarrow A'$ of log abelian varieties with constant degeneration, which fits into a commutative diagram 
\[\xymatrix{
0\ar[r] &Y\ar[r]\ar[d]^{f_{-1}} &G^{(Y)}_{\mr{log}}\ar[r]\ar[d]^{\tilde{f}} &A\ar[r]\ar[d]^f &0 \\
0\ar[r] &Y'\ar[r]  &G'^{(Y')}_{\mr{log}}\ar[r]  &A'\ar[r] &0
}\]
with exact rows. It then follows from Lemma \ref{lem5.3} that $f$ is an isogeny with kernel $F$.
\end{proof}

\subsection{Serre-Tate theorem}
In this subsection, we closely follow \cite{katz1}.

Let $R$ be a ring which is killed by an integer $N\geq 1$, and $I$ an ideal of $R$ satisfying $I^{v+1}=0$. Let $S$ be an fs log scheme with its underlying scheme $\Spec R$, $(\mr{fs}/S)$ the category of fs log schemes over $S$, and $(\mr{fsAff}/S)$ the full subcategory of $(\mr{fs}/S)$ consisting of fs log schemes whose underlying scheme is affine. We denote $R/I$ (resp. $\Spec(R/I)$) by $R_0$ (resp. $S_0$), and we endow $S_0$ with the induced log structure from $S$. 

\begin{defn}
Let $G$ be a functor from $(\mr{fsAff}/S)$ to the category of abelian groups. The subgroup functor $G_I$ of $G$ is defined by 
$$G_I(U)=\mr{ker}(G(U)\rightarrow G(U\otimes_RR_0))$$
for $U\in (\mr{fsAff}/S)$, where $U\otimes_RR_0$ is endowed with the induced log structure from $U$. 

The subgroup functor $\hat{G}$ of $G$ is defined by
$$\hat{G}(U)=\mr{ker}(G(U)\rightarrow G(U^{\mr{red}})),$$
for $U\in (\mr{fsAff}/S)$, where $U^{\mr{red}}$ is the reduced closed subscheme of $U$ endowed with the induced log structure from $U$.
\end{defn}

\begin{lem}\label{lem5.4}
Let $G$ be a commutative formal Lie group over $R$, and we endow $G$ with the log structure induced from $S$. Then the subgroup functor $G_I$ is killed by $N^v$.
\end{lem}
\begin{proof}
See \cite[Lem. 1.1.1]{katz1}.
\end{proof}

\begin{lem}\label{lem5.5}
Let $G$ be a sheaf of abelian groups on $(\mr{fs}/S)_\mr{{kfl}}$ such that $\hat{G}$ is Kummer log flat locally representable by a classical formal Lie group over $R$. Then $N^v$ kills $G_I$.
\end{lem}
\begin{proof}
See \cite[Lem. 1.1.2]{katz1}.
\end{proof}

\begin{lem}\label{lem5.6}
Let $G$ and $H$ be two sheaves of abelian groups on $(\mr{fs}/S)_\mr{{kfl}}$ such that:
\begin{enumerate}[(a)]
\item $G$ is $N$-divisible;
\item $\hat{H}$ is Kummer log flat locally representable by a classical formal Lie group;
\item $H$ is formally log smooth.
\end{enumerate}
Let $G_0$ (resp. $H_0$) denote the pullback of $G$ (resp. $H$) to $(\mr{fs}/S_0)_\mr{{kfl}}$. Then we have:
\begin{enumerate}[(1)]
\item the groups $\mr{Hom}_S(G,H)$ and $\mr{Hom}_{S_0}(G_0,H_0)$ have no $N$-torsion;
\item The natural map ``reduction mod $I$'' $\mr{Hom}_S(G,H)\rightarrow\mr{Hom}_{S_0}(G_0,H_0)$ is injective;
\item for any homomorphism $f_0:G_0\rightarrow H_0$, there exists a unique homomorphism $``N^vf'':G\rightarrow H$ which lifts $N^vf_0$;
\item in order that a homomorphism $f_0:G_0\rightarrow H_0$ lifts to a (necessarily unique) homomorphism $f:G\rightarrow H$, it is necessary and sufficient that the homomorphism $``N^vf'':G\rightarrow H$ annihilates the subgroup $G[N^v]:=\mr{ker}(G\xrightarrow{N^v}G)$ of $G$.
\end{enumerate}
\end{lem}
\begin{proof}
Properties (1), (2), (4) follow from Lemma \ref{lem5.5} just as in \cite[Lem. 1.1.3]{katz1} . Only (3) requires a bit more care. Let $T$ be an fs log scheme over $S$, and let $T_{0}$ be the closed subscheme defined by $I$ together with the induced log structure from $T$.  We will construct a map $G(T)\to H(T)$ lifting $N^{v}f_{0}$ and its uniqueness will follow from (2).

For any  $t\in G(T)$, let $t_{0}\in G(T_{0})$ be its reduction mod $I$. Then by formal log smoothness of $H$ there is an \'etale  cover $\tilde{T}\to T$ such that the image of $f_{0}(t_{0})$ in $H(\tilde{T}_{0})$ lifts to some element $\tilde{h}\in H(\tilde{T})$. Now $N^{v}\tilde{h}$ is unique, so in particular on the overlap $\tilde{T}\times_{T} \tilde{T}$ we have $\mr{pr}_{1}^{*}(N^{v}\tilde{h})=\mr{pr}_{2}^{*}(N^{v}\tilde{h})$. Hence $N^{v}\tilde{h}$ descends to a unique $h\in H(T)$. One readily checks that this construction is natural in $T$.
\end{proof}

Now let $N$ be $p^r$ for a prime number $p$ and a positive integer $r$. Let $A$ be a log abelian variety with constant degeneration over $S$, and $A[p^{\infty}]$ the log $p$-divisible group associated to $A$. We are going to show that both $A$ and $A[p^{\infty}]$ satisfy the conditions (a)-(c) from Lemma \ref{lem5.6}.

\begin{prop}\label{prop5.2}
The conditions (a)-(c) from Lemma \ref{lem5.6} hold for $A$.
\end{prop}
\begin{proof}
The multiplication by $n$ map $n_A:A\rightarrow A$ is an isogeny by \cite[Cor. 2.1]{zha4} for any positive integer $n$. In particular, $A$ is $p$-divisible.

By \cite[Thm. 4.1]{k-k-n4}, $A$ is log smooth over $S$. In particular it is formally log smooth over $S$.

We are left with checking the condition (b). Assume that $A\cong G_{\mr{log}}^{(Y)}/Y$ for a pointwise polarizable log 1-motive $[Y\to G_{\mr{log}}]$ over $S$. By \cite[Thm. 2.1]{zha3}, $A$ fits into the following short exact sequence
\begin{equation}\label{eq3.1}
0\rightarrow G\rightarrow A\rightarrow \mc{H}om_S(X,\Gml/\Gm)^{(Y)}/Y\rightarrow 0
\end{equation}
of sheaves of abelian groups over $(\mr{fs}/S)_\mr{{kfl}}$. We abbreviate $\mc{H}om_S(X,\Gml/\Gm)^{(Y)}$ as $\mc{Q}$. Let $U\in(\mr{fsAff}/S)$, we have the following two commutative diagrams
\begin{equation}\label{eq3.2}
\xymatrix{
0\ar[r] &G(U)\ar[r]\ar[d] &A(U)\ar[r]\ar[d] &(\mc{Q}/Y)(U)\ar[d]   \\
0\ar[r] &G(U^{\mr{red}})\ar[r] &A(U^{\mr{red}})\ar[r] &(\mc{Q}/Y)(U^{\mr{red}})
}
\end{equation}
and 
\begin{equation}\label{eq3.3}
\xymatrix{
0\ar[r] &Y(U)\ar[r]\ar[d] &\mc{Q}(U)\ar[r]\ar[d] &(\mc{Q}/Y)(U)\ar[r]\ar[d] &H^1_{\mr{kfl}}(U,Y)\ar[d]  \\
0\ar[r] &Y(U^{\mr{red}})\ar[r] &\mc{Q}(U^{\mr{red}})\ar[r] &(\mc{Q}/Y)(U^{\mr{red}})\ar[r] &H^1_{\mr{kfl}}(U^{\mr{red}},Y)
}
\end{equation}
with exact rows. By \cite[Lem. 2.4]{zha3}\footnote{The proof of \cite[Lem. 2.4]{zha3}
makes use of fpqc descent of schemes which probably does not always hold and deserves a precise reference (see \cite[\href{https://stacks.math.columbia.edu/tag/0APK}{Tag 0APK}]{stacks-project} for the situation there). Alternatively, see \cite[Thm. 1.3]{zha8}.}, we have 
\[H^1_{\mr{kfl}}(U,Y)\cong H^1_{\mr{fl}}(U,Y) \text{ and } H^1_{\mr{kfl}}(U^{\mr{red}},Y)\cong H^1_{\mr{fl}}(U^{\mr{red}},Y).\]
Since $Y$ is a smooth group scheme over $S$, we further get $H^1_{\mr{fl}}(U,Y)\cong H^1_{\mr{\acute{e}t}}(U,Y)$ and $H^1_{\mr{fl}}(U^{\mr{red}},Y)\cong H^1_{\mr{\acute{e}t}}(U^{\mr{red}},Y)$. Since the small \'etale site of $U$ and the small \'etale site of $U^{\mr{red}}$ are equivalent, we get $H^1_{\mr{\acute{e}t}}(U,Y)\cong H^1_{\mr{\acute{e}t}}(U^{\mr{red}},Y)$ and $Y(U)\cong Y(U^{\mr{red}})$. Hence the left vertical map and the right vertical map of (\ref{eq3.3}) are both isomorphisms. Since the formation of $\Gml/\Gm$ is compatible with strict base change of log schemes, we get that the canonical map $\mc{Q}(U)\rightarrow\mc{Q}(U^{\mr{red}})$ is an isomorphism. Applying the five lemma to (\ref{eq3.3}), we get that the canonical map $\mc{Q}/Y(U)\rightarrow \mc{Q}/Y(U^{\mr{red}})$ is injective. By diagram (\ref{eq3.2}), we have that the canonical map $\hat{G}(U)\rightarrow\hat{A}(U)$ is an isomorphism. Since $\hat{G}$ is representable by a formal Lie group, so is $\hat{A}$.
\end{proof}

\begin{prop}\label{prop5.3}
Let $S_0$ be an fs log scheme where $p$ is locally nilpotent. Let $\mathbf{M}_0=[Y_0\to G_{0,\mr{log}}]$ be a log 1-motive over $S_0$. Let $S_0\to S$ be strict closed immersion of fs log schemes defined by a nilpotent sheaf of ideals, and let $H$ be a log $p$-divisible group over $S$ such that $H_0:=H\times_SS_0= \mathbf{M}_0[p^\infty]$. Then
 $\hat{H}_0$ (resp. $\hat{H}$) is locally a formal Lie group over $S_0$ (resp. over $S$).
\end{prop}
\begin{proof}
The log $p$-divisible group $H_0$ fits into a short exact sequence
\[0\to G_0[p^\infty]\to H_0\to Y_0\otimes_Z\Q_p/\Z_p\to 0.\]
Since $Y_0\otimes_Z\Q_p/\Z_p$ is classical \'etale, the subgroup functor $\hat{H}_0$ of $H_0$ coincides with $\hat{G}_0=\widehat{G_0[p^\infty]}$ which is locally a formal Lie group.

Since the small \'etale site of $S_0$ is equivalent to the small \'etale site of $S$, $Y_0$ extends to a locally constant sheaf $Y$ over $S$. Let $H_n:=H[p^n]$ and $H_{0,n}:=H_0[p^n]$. Consider the commutative diagram
\[\xymatrix{
H_{0,n}\times_{S_0}H_{0,n}\ar[r]^-{\bar{m}_n}\ar@{^(->}[d] &H_{0,n}\ar[r]^-{i_n\circ\bar{p}_n}\ar[d] &Y/p^nY\ar[d] \\
H_{n}\times_{S}H_{n}\ar[r]_-{m_n}\ar@{-->}[rru]^{P_n} &H_n\ar[r]\ar@{..>}[ru]_{p_n} &S
}\]
of solids arrows, where $\bar{m}_n$ (resp. $m_n$) is the group law of $H_{0,n}$ (resp. $H_n$), $\bar{p}_n$ is the composition $H_{0,n}=T_{p^n}(\mathbf{M}_0)\to Y_0/p^nY_0$, and $i_n$ is the closed immersion $Y_0/p^nY_0\hookrightarrow Y/p^nY$. Since $Y/p^nY$ is log \'etale (in fact even classical \'etale) over $S$ and $H_{0,n}\to H_n$ is a strict closed immersion defined by a nilpotent ideal, $i_n\circ\bar{p}_n$ and $i_n\circ\bar{p}_n\circ\bar{m}_n$ lift to a unique map $p_n$ and $P_n$ respectively such that the diagram remains commutative. Let $\mu_n$ be the group law of $Y/p^nY$. Since $(i_n\circ\bar{\mu}_n)\circ \bar{p}_n\times \bar{p}_n=(i_n\circ\bar{p}_n)\circ \bar{m}_n$, the outer square of the commutative diagram
\[\xymatrix{
H_{0,n}\times_{S_0}H_{0,n}\ar[r]^-{\bar{p}_n\times \bar{p}_n}\ar@{^(->}[d] &Y_0/p^nY_0\times_{S_0}Y_0/p^nY_0\ar[r]^-{i_n\circ\bar{\mu}_n}\ar[d] &Y/p^nY\ar[d] \\
H_{n}\times_{S}H_{n}\ar[r]_-{p_n\times p_n}\ar@{-->}[rru]^{P_n} &Y/p^nY\times_SY/p^nY\ar[r]\ar[ru]_{\mu_n} &S
}\]
of solid arrows coincides with the outer square of the last diagram. Therefore the map $P_n$ from last diagram is also the unique lifting of $(i_n\circ\bar{\mu}_n)\circ \bar{p}_n\times \bar{p}_n$ such that the diagram remains commutative. It follows that $\mu_n\circ(p_n\times p_n)=P_n=p_n\circ m_n$. Therefore $p_n$ is a homomorphism of group log schemes. By the uniqueness of $p_n$, it is also clear that $(p_n)_{n\geq1}$ gives rise to a homomorphism 
\[p:H\to Y\otimes_\Z\Q_p/\Z_p.\]

Let $H'_n:=\mr{ker}(H_n\xrightarrow{p_n}Y/p^nY)$. Clearly $H'_n$ is representable. Since $S_0\to S$ is a strict nil immersion and 
\[(H'_n)\times_SS_0=G_0[p^n]\in(\mr{fin}/S_0)_{\mr{c}},\]
$H'_n$ is strict over $S$. Since $G_0[p^n]$ is classically flat, in particular log flat, $(p_n)_0$ is log flat over $S_0$ by descent. Since $H_n$ is log flat over $S$, $p_n$ is log flat by Corollary \ref{corB.1}. Hence $H_n'=\mr{ker}(p_n)$ is log flat. Since it is also strict over $S$, it is even classically flat. Moreover, $H_n'\to S$ agrees with \[H_n'=H_n\times_{Y/p^nY}S\to Z\to H_n\to S,\]
where $Z$ denotes the fiber product of $H_n\rightarrow Y/p^nY\xleftarrow{e} S$ in the category of schemes. The first map is finite by \cite[Chap. III, Cor. 2.1.6]{ogu1}. Since the unit section $e$ of $Y/p^nY$ is a closed immersion, the second map is also finite. The last map is finite by Proposition \ref{prop1.1}. It follows that $H_n'\in(\mr{fin}/S)_{\mr{c}}$.

Thus we get a classical $p$-divisible group $H':=\varinjlim_nH'_n$ such that $H'\times_SS_0=G_0[p^\infty]$. Using the exact sequence
\[0\to H'\to H\xrightarrow{p} Y\otimes_\Z\Q_p/\Z_p,\]
we get that $\hat{H}=\hat{H'}$ is locally a formal Lie group.
\end{proof}

\begin{prop}\label{prop5.4}
Let $S_0$, $S$, $H_0$, and $H$ be as in Proposition \ref{prop5.3}. Then $H_0$ and $H$ satisfies the conditions (a)-(c) from Lemma \ref{lem5.6}.
\end{prop}
\begin{proof}
Condition (a) is trivial, condition (b) follows from Proposition \ref{prop5.3}, and condition (c) follows from Theorem \ref{thm2.4}. 
\end{proof}

We denote by $\mr{LAVwCD}_S$ the category of log abelian varieties with constant degeneration over $S$, and by $\mr{Def}(S,S_0)$ the category of triples 
$$(A_0,H,\varepsilon)$$
consisting of a log abelian variety with constant degeneration $A_0$ over $S_0$, an object $H$ of $(\text{$p$-div}/S)^{\mr{log}}_{\mr{d}}$, and an isomorphism 
$\varepsilon:H_0\xrightarrow{\cong}A_0[p^{\infty}]$ of log $p$-divisible groups over $S_0$.

\begin{thm}[Serre-Tate for LAVwCD]\label{thm5.1}
Let $R,R_0,S,S_0,I, N$, and $v$ be as in the beginning of this subsection. We further assume that $N$ is a power of a prime number $p$. Then the canonical functor
$$\mr{LAVwCD}_S\rightarrow\mr{Def}(S,S_0), A\mapsto (A_0,A[p^{\infty}],\mr{id}_{A_0[p^{\infty}]})$$
where $A_0$ denotes the base change of $A$ to $S_0$, is an equivalence of categories.
\end{thm}
\begin{proof}
Step 1: Full faithfulness. Let $A, A'$ be log abelian varieties over $S$, $f_{p^{\infty}}:A[p^{\infty}]\rightarrow A'[p^{\infty}]$ a homomorphism of log $p$-divisible groups over $S$, and $f_0:A_0\rightarrow A'_0$ a homomorphism of log abelian varieties over $S_0$ such that $f_0[p^{\infty}]$ coincides with $f_{p^{\infty}}\times_SS_0$. To finish the proof of full faithfulness, it suffices to show that there exists a unique homomorphism $f:A\rightarrow A'$ which induces both $f_{p^{\infty}}$ and $f_0$. 
The uniqueness follows easily from Proposition \ref{prop5.2} and Lemma \ref{lem5.6} (2).

To construct the homomorphism $f$, consider the canonical lifting $``N^vf":A\rightarrow A'$ of $N^vf_0$ guaranteed by Lemma \ref{lem5.6} (3). Since $``N^vf"$ lifts $N^vf_0$, the induced map $``N^vf"[p^{\infty}]$ on log $p$-divisible groups lifts $N^v(f_0[p^{\infty}])$. By part (2) of Lemma \ref{lem5.6}, we have  $``N^vf"[p^{\infty}]=N^vf_{p^{\infty}}$. Hence $``N^vf "$ kills $A[N^v]$. By part (4) of Lemma \ref{lem5.6}, there exists a homomorphism $f:A\rightarrow A'$ lifting $f_0$ such that $``N^vf"=N^vf$. Since $f$ lifting $f_0$ implies $f[p^{\infty}]$ lifting $f_0[p^{\infty}]$, we get $f[p^{\infty}]=f_{p^{\infty}}$ again by part (2) of Lemma \ref{lem5.6}. This finishes the proof of the full faithfulness.

Step 2: Essential surjectivity. Let $(A_0,H,\alpha_{0})$ be an object of $\mr{Def}(S,S_0)$, we look for a log abelian variety $A$ which gives rise to $(A_0,H,\alpha_{0})$ up to isomorphism. By \cite[Thm. 1.6]{k-k-n7} and \cite[Cor. 3.11]{kat1}, $A_0$ lifts to a log abelian variety with constant degeneration $A'$ over $S$. The isomorphism 
\[\alpha_0:A'_0:=A'\times_SS_0\xrightarrow{\cong} A_0\]
of log abelian varieties over $S_0$ induces an isomorphism 
\[\alpha_0[p^{\infty}]:A'_0[p^{\infty}]\rightarrow A_0[p^{\infty}]\]
of log $p$-divisible groups over $S_0$. By part (3) of Lemma \ref{lem5.4}, $N^v\alpha_0[p^{\infty}]$ has a unique lifting to a homomorphism 
\[``N^v\alpha[p^{\infty}]":A'[p^{\infty}]\rightarrow H\]
of log $p$-divisible groups over $S$. The same procedure applied to $\beta_0:=\alpha_0^{-1}$, we get a homomorphism 
$$``N^v\beta[p^{\infty}]":H\rightarrow A'[p^{\infty}]$$
of log $p$-divisible groups over $S$. Since the composition $(N^v\beta_0[p^{\infty}])\circ(N^v\alpha_0[p^{\infty}])$ (resp. $(N^v\alpha_0[p^{\infty}])\circ (N^v\beta_0[p^{\infty}])$)
is the multiplication by $N^{2v}$ map on $A'_0[p^{\infty}]$ (resp. $H_0$), the composition 
$``N^v\beta[p^{\infty}]"\circ ``N^v\alpha[p^{\infty}]"$ (resp. $``N^v\alpha[p^{\infty}]"\circ ``N^v\beta[p^{\infty}]"$) is the multiplication by $N^{2v}$ map on $A'[p^{\infty}]$ (resp. $H$). Let $K:=\mr{ker}(``N^v\alpha[p^{\infty}]")$ and $Q:=\mr{ker}(``N^v\beta[p^{\infty}]")$. Thus we get a short exact sequence
\[0\to K\to A[N^{2v}]\xrightarrow{u} Q\to 0\]
using the kernel-cokernel exact sequences for the compositions 
\[``N^v\beta[p^{\infty}]"\circ ``N^v\alpha[p^{\infty}]"=N^{2v}\]
and 
\[``N^v\alpha[p^{\infty}]"\circ ``N^v\beta[p^{\infty}]"=N^{2v}.\]

\begin{claim}
$K=\mr{ker}(``N^v\alpha[p^{\infty}]")$ lies in $(\mr{fin}/S)_{\mr{r}}$.
\end{claim}
Since $K$ is also the kernel of $A'[N^{2v}]\to H[N^{2v}]$, it is representable by an fs log scheme. Clearly both $K_0:=K\times_SS_0$ and $Q_0:=Q\times_SS_0$ lie in $(\mr{fin}/S_0)_{\mr{r}}$. Then $u_0:A'_0[N^{2v}]\to Q_0$ as a $K_0$-torsor is log flat by Kummer log flat descent (see \cite[Thm. 0.1]{i-n-t1}). By the log flatness of $A'[N^{2v}]\to S$ and  Corollary \ref{corB.1}, $u$ is log flat. From this it then follows that $K$ is log flat over $S$. Moreover, the Kummer-property extends to infinitesimal lifts by Lemma \ref{lem5.7} below. Therefore, $K$ is also of Kummer type over $S$. Being the kernel of $A'[N^{2v}]\to H[N^{2v}]$, the group $K$ is also finite over $S$, by the same argument as in the proof of the finitenss of the group $H_n'$ in the proof of Proposition \ref{prop5.3}. But then, $K$ being finite Kummer log flat over $S$, one has $K\in (\mr{fin}/S)_{\mr{r}}$ (see Proposition \ref{prop1.1}). This finishes the proof of the claim.

Let $[Y\to G_{\mr{log}}]$ be the log 1-motive corresponding to $A'$. Let 
\[K':=\mr{ker}(K\to A'[N^{2v}]\to Y/N^{2v}Y).\]
Since $K_0=\mr{ker}(N^v\alpha_0[p^\infty])=A'_0[N^v]$, we have 
\[K'_0=\mr{ker}(A'_0[N^v]\to Y/N^{2v}Y)=G_0[N^v].\]
Replace $u$ by $K\to Y/N^{2v}Y$ and $K$ by $K'$, the same proof as that of the claim implies that $K'\in (\mr{fin}/S)_{\mr{r}}$. Now applying Proposition \ref{prop5.1}, we get that $A'\to A'/K$ is an isogeny of log abelian varieties with constant degeneration with kernel $K$. Let $A:=A'/K$. Apparently 
\[A\times_SS_0=A'_0/K_0=A'_0/A'_0[N^v]= A'_0\cong A_0,\]
and $``N^v\alpha[p^{\infty}]"$ induces an isomorphism $A[p^\infty]\xrightarrow{\cong}H$. This finishes the proof of the essential surjectivity.
\end{proof}

In the proof above we have used the following lemma.
\begin{lem}\label{lem5.7}
Assume that $S_{0}\subset S$ is a strict nilpotent thickening of fs log schemes and let $f:X\to Y$ be a morphism of fs log schemes over $S$, such that $f_{0}=f\times_{S} S_{0}$ is Kummer. Then $f$ is also Kummer.
\end{lem}
\begin{proof}
For any log scheme $Z$, we denote $M_Z/\mc{O}_Z^\times$ by $\overline{M}_Z$.

Let $x$ be any point of $X$, and let $y=f(x)$. It suffices to show that the canonical map $\overline{M}_{Y,\bar{y}}\to \overline{M}_{X,\bar{x}}$ of fs monoids is Kummer. Let $X_0:=X\times_SS_0$ and $Y_0:=Y\times_SS_0$. Consider the following commutative diagram
\[\xymatrix{
X_0\ar[r]\ar[d]_{f_0} &X\ar[d]^f \\
Y_0\ar[r] &Y
}\]
with horizontal maps strict. This square induces another commutative square
\[\xymatrix{
\overline{M}_{X_0,\bar{x}} &\overline{M}_{X,\bar{x}}\ar[l]_{\cong}  \\
\overline{M}_{Y_0,\bar{y}}\ar[u] &\overline{M}_{Y,\bar{y}}\ar[l]_{\cong}\ar[u]
}.\]
Since $f_0$ is Kummer, the left vertical map in the above square is Kummer. If follows that the right vertical map is also Kummer.
\end{proof}

\appendix

\section{A lemma on finite flat group schemes}
\begin{lem}\label{lemA.1}
Let $S$ be a scheme, $F$ a finite flat group scheme over $S$ which is of multiplicative type, and $F'$ a finite flat group scheme over $S$. Then the fppf sheaf $H:=\mc{H}om_S(F,F')$ is representable by an \'etale quasi-finite separated group scheme over $S$.
\end{lem}
\begin{proof}
Let $0\rightarrow F'\rightarrow G_1\rightarrow G_2\rightarrow 0$ be the canonical smooth resolution of $F'$ with $G_1,G_2$ affine smooth group schemes over $S$, see \cite[Thm. A.5]{mil2}. Then we have an exact sequence
$$0\rightarrow \mc{H}om_S(F,F')\rightarrow \mc{H}om_S(F,G_1)\xrightarrow{\alpha} \mc{H}om_S(F,G_2)$$
of fppf sheaves of abelian groups over $S$. By \cite[Exp. XI, Cor. 4.2]{sga3-2}, the sheaves $\mc{H}om_S(F,G_i)$ for $i=1,2$ are representable by smooth separated group schemes over $S$. Hence $H=\mc{H}om_S(F,F')$ is the kernel of a morphism of representable sheaves and thus representable itself. Let $n$ be a positive integer which kills $F$. We then have 
\[\mc{H}om_S(F,G_i)=\mc{H}om_S(F,G_i)[n],\]
where $\mc{H}om_S(F,G_i)[n]$ denotes the $n$-torsion subgroup scheme of $\mc{H}om_S(F,G_i)$. The fibers of the group scheme $\mc{H}om_S(F,G_i)$ over $S$ are finite by the structure theorem \cite[Exp. XVII, Thm. 7.2.1]{sga3-2} of commutative group schemes and \cite[Exp. XVII, Prop. 2.4]{sga3-2}. It follows that the group schemes $\mc{H}om_S(F,G_i)$ for $i=1,2$ are \'etale and quasi-finite over $S$. By \cite[Cor. 17.3.5]{egaIV-4}, $\alpha$ is \'etale. It is also separated. It follows that $H$ is \'etale separated and quasi-finite. This finishes the proof.
\end{proof}

\section{Fiberwise flatness for algebraic stacks}
\begin{lem}\label{applem}
Let $\mathcal{S}$ be a locally noetherian algebraic stack. Let $f:\mathcal{X}\to \mathcal{Y}$ be a 1-morphism of locally noetherian algebraic stacks over $\mathcal{S}$, where $\mathcal{X}$ is flat over $\mathcal{S}$. Assume that $f_{s}:\mathcal{X}_{s}\to \mathcal{Y}_{s}$ is flat for all points (valued in fields) $s:\Spec(k(s))\to \mathcal{S}$. Then $f$ is flat.
\end{lem}
\begin{proof}
As being locally noetherian and flatness of a morphism are local for the smooth topology, everything reduces to the case of schemes, where the relevant statement is \cite[Tag 039D]{stacks-project}.

More precisely, we have a commutative diagram
\begin{center}
$\xymatrix{
X\ar[r]^{\tilde{f}}\ar[d] & Y\ar[r]\ar[d] & S\ar[d] \\
\mathcal{X}\ar[r]^{f} & \mathcal{Y}\ar[r] & \mathcal{S}},
$
\end{center}
where the vertical morphisms are smooth covers by locally noetherian schemes. We then let $s':\Spec(k(s))\to S$ be a point of $S$. It maps to a point $s$ of $\mathcal{S}$. We then get a commutative diagram
\begin{center}
$\xymatrix{
X_{s'}\ar[r]^{\tilde{f}_{s'}}\ar[d] & Y_{s'}\ar[d]  \\
\mathcal{X}_{s}\ar[r]^{f_{s}} & \mathcal{Y}_{s}}.
$
\end{center}
where again the vertical maps are smooth coverings by schemes. Now $f_{s}$ is flat by assumption. But this is the same thing as saying that $\tilde{f}_{s'}$ is flat (see for example \cite[\href{https://stacks.math.columbia.edu/tag/06FN}{Tag 06FN}]{stacks-project} and note that flatness is local on the source-and-target). Hence $\tilde{f}_{s'}$ is flat for all $s'\in S$. So, by the fiberwise flatness criterion for schemes, $\tilde{f}$ is flat as well. This in turn implies that $f$ is flat. 
\end{proof} 

\begin{cor}\label{corB.1}
Let $S_0$ be a  noetherian fs log scheme, and let $S_0\to S$ be a strict closed immersion defined by a nilpotent sheaf of ideals $I$. Let $f:G\to H$ be a morphism of noetherian fs log schemes over $S$ with $G\to S$ log flat. Assume that $f_0:=f\times_SS_0:G_0\to H_0$ is log flat. Then $f$ is log flat. 
\end{cor}
\begin{proof}
Consider the following diagram of algebraic stacks
\begin{center}
$\xymatrix{
\mathcal{L}og_{G}\ar[rr]^{\mathcal{L}og(f)}\ar[dr] && \mathcal{L}og_{H}\ar[dl] \\
& \mathcal{L}og_{S} &}.
$
\end{center}
Here, if $Z$ is a log scheme, $\mathcal{L}og_{Z}$ denotes the stack of log structures introduced in \cite[\S 1]{ols1}. It is an algebraic stack. If $Z$ is locally noetherian, then so is $\mathcal{L}og_Z$ (by \cite[Cor. 5.25]{ols1}). Furthermore a morphism $f:Z\to Z'$ of fine log schemes is log flat if and only if the induced morphism $\mathcal{L}og(f):\mathcal{L}og_{Z}\to \mathcal{L}og_{Z'}$ is flat (see \cite[Theorem 4.6 + Remark 4.7]{ols1}).

Moreover, by \cite[Proposition 3.20]{ols1}, the diagram above reduces modulo $I$ to 
\begin{center}
$\xymatrix{
\mathcal{L}og_{G_0}\ar[rr]^{\mathcal{L}og(f_{0})}\ar[dr] && \mathcal{L}og_{H_{0}}\ar[dl] \\
& \mathcal{L}og_{S_0} &}.
$
\end{center}
Now, as $f_{0}$ is log flat, $\mathcal{L}og(f_{0})$ is flat. Thus, by the fiberwise flatness criterion for algebraic stacks (see Lemma \ref{applem}), $\mathcal{L}og(f)$ is flat as well, which means that $f$ is log flat. 
\end{proof}

\section{Kummer log flat descent of finiteness}
Following \cite[Def. 1.9]{nak1}, we say a morphism $f:X\to S$ of fs log scheme is finite (resp. quasi-finite, resp. separated, resp. proper, resp. universally closed, resp. locally of finite presentation, surjective), if the underlying morphism of schemes is so. 

\begin{rmk}\label{rmkC.1}
According to \cite[Lem. 4.8]{ols1}, $f:X\to S$ is locally of finite presentation if and only if the induced morphism $\mc{L}og(f):\mc{L}og_X\to \mc{L}og_S$ of algebraic stacks is locally of finite presentation, where $\mc{L}og_T$, for a log scheme $T$, is the stack of log structures as in \cite[\S 1]{ols1}.
\end{rmk}

By \cite[\href{https://stacks.math.columbia.edu/tag/02LA}{Tag 02LA}]{stacks-project}, the finiteness of morphisms of schemes descends in the fpqc topology. According to \cite[Thm. 7.1]{kat2}, the descent of the finiteness of morphisms of fs log schemes in the Kummer flat topology has been shown in Tani's thesis \cite{tan1}. However the thesis is in Japanese, so we present a proof here.

Let $X\to S$ and $T\to S$ be two morphisms of fs log schemes. We denote by $\mathring{X}\times_{\mathring{S}}\mathring{T}$ the fiber product of the underlying schemes. Note that in this paper $X\times_ST$ denotes the fiber product in the category of fs log schemes.
\begin{prop}\label{propC.1}
Let $f:X\to S$ be a morphism of fs log schemes, and $g:S'\to S$ a Kummer log flat cover, see Definition \ref{defn1.1}. Let $f':X':=X\times_SS'\rightarrow S'$ be the base change of $f$. Then $f$ is finite if and only if $f'$ is finite.
\end{prop}
\begin{proof}
If $f$ is finite, so is $f'$ by \cite[\S 1.10]{nak1}. 

Conversely assume that $f'$ is finite. To show that $f$ is finite, we need to show that $f$ is (1) quasi-finite, (2) universally closed, (3) separated, and (4) of finite type. 

(1) The quasi-finiteness of $f$ follows from Nakayama's ``fourth point lemma'' (see \cite[\S 2.2.2]{nak1}) and the finiteness of $f'$. 

(2) To show $f$ is universally closed, it suffices to show that for any strict morphism $U\to S$, the base change $X\times_SU\to U$ is closed. Consider the following diagram
$$\xymatrix{
X'\ar[d] &X'\times_SU\ar[l]\ar[r]\ar[d] &X\times_SU\ar[d]  \\
S' &S'\times_SU\ar[l]\ar[r] &U
}$$
with the two squares cartesian. Since $X'\to S'$ is universally closed and $S'\times_SU\to S'$ is strict, the map $X'\times_SU\to S'\times_SU$ is closed. Since $S'\to S$ is a Kummer log flat cover, so are $S'\times_SU\to U$ and $X'\times_SU\to X\times_SU$. By \cite[Prop. 2.5]{kat2}, the underlying topological maps of $S'\times_SU\to U$ and $X'\times_SU\to X\times_SU$ are open and surjective. 

Let $(\mr{logSch}/S)$ be the category of log schemes over $S$, and let $(\mr{int}/S)$ be the full subcategory consisting of integral log schemes over $S$. Recall that $(\mr{fs}/S)$ is the full subcategory consisting of fs log schemes over $S$. We denote by $Z$ the fiber product of $S'\times_SU$ and $X\times_SU$ over $U$ in $(\mr{logSch}/S)$. Then the integral log scheme $Z^{\mr{int}}$ associated to $Z$ is the fiber product of $S'\times_SU$ and $X\times_SU$ over $U$ in $(\mr{int}/S)$, and the fs log scheme $Z^{\mr{fs}}:=(Z^{\mr{int}})^{\mr{sat}}$ associated to $Z^{\mr{int}}$ is the fiber product of $S'\times_SU$ and $X\times_SU$ over $U$ in $(\mr{fs}/S)$, see \cite[Chap. III, Prop. 2.1.5, Cor. 2.1.6]{ogu1}. Note that the fiber product $Z^{\mr{fs}}$ of $S'\times_SU$ and $X\times_SU$ over $U$ in $(\mr{fs}/S)$ is nothing but the fiber product $X'\times_SU$ (in $(\mr{fs}/S)$). The map $X'\times_S U=Z^{\mr{fs}}\to Z^{\mr{int}}$ is finite and surjective by \cite[Chap. III, Prop. 2.1.5 part 2]{ogu1}, and the map $Z^{\mr{int}}\to Z$ is a nil immersion by \cite[Chap. III, Cor. 2.2.4]{ogu1}. Now we put the information useful for us in the following commutative diagram
\[\xymatrix{
X'\times_SU\ar@/^2pc/[rrrdd]^-{\text{open surjective}}_{b}\ar@/_2pc/[rrddd]_{\mr{closed}}^a\ar[rd]^-{\text{surjective}}_c \\
&Z^{\mr{int}}\ar[rd]^-{\text{homeomorphism}}_d \\
&&Z\ar[r]^-\alpha_-{\text{surj.}}\ar[d]_\beta  &X\times_SU\ar[d]^\delta  \\
&&S'\times_SU\ar[r]_-{\text{open surj.}}^{\gamma} &U
}\]
of topological spaces. To prove that $X\times_SU\to U$ is closed, it suffices to show that $\delta(V)$ is closed for any closed subset $V$ of $X\times_SU$. Since $c$ is surjective, one can check that $\beta(\alpha^{-1}(V))=a(b^{-1}(V))$. Since the map $a$ is closed, the set $\beta(\alpha^{-1}(V))$ is closed. Since $\alpha$ is surjective, one sees that $\delta(V)=\gamma(\beta(\alpha^{-1}(V)))$. By the cartesian property of the inner square in the category of schemes, one sees that $\gamma^{-1}(\delta(V))=\beta(\alpha^{-1}(V))$ by \cite[Lem. 4.28]{g-w1}. Since $\gamma$ is open and $\beta(\alpha^{-1}(V))$ is closed, we get that $\delta(V)$ is closed.

(3) Let $\Delta_{X'/S'}:X'\rightarrow X'\times_{S'}X'$ (resp. $\Delta_{\mathring{X}'/\mathring{S}'}:\mathring{X}'\rightarrow \mathring{X}'\times_{\mathring{S}'}\mathring{X}'$) be the diagonal map in the category of fs log schemes (resp. category of schemes). Similarly for $\Delta_{X/S}$ and $\Delta_{\mathring{X}/\mathring{S}}$. We have $X'\times_{S'}X'=(X\times_SX)\times_SS'$ (resp. $\mathring{X}'\times_{\mathring{S}'}\mathring{X}'=(\mathring{X}\times_{\mathring{S}}\mathring{X})\times_{\mathring{S}}\mathring{S}'$) and $\Delta_{X'/S'}=\Delta_{X/S}\times_SS'$ (resp. $\Delta_{\mathring{X}'/\mathring{S}'}=\Delta_{\mathring{X}/\mathring{S}}\times_{\mathring{S}}\mathring{S}'$). In the commutative diagram 
$$\xymatrix{
X'\ar[r]^-{\Delta_{X'/S'}}\ar[rd]_-{\Delta_{\mathring{X}'/\mathring{S}'}} &X'\times_{S'}X'\ar[d]  \\
&\mathring{X}'\times_{\mathring{S}'}\mathring{X}'
},$$
the map $\Delta_{\mathring{X}'/\mathring{S}'}$ is a closed immersion by the separatedness of $X'\to S'$, the map $X'\times_{S'}X'\to \mathring{X}'\times_{\mathring{S}'}\mathring{X}'$ is finite by \cite[\S 1.10]{nak1}, hence $\Delta_{X'/S'}$ is a closed immersion. In particular $\Delta_{X'/S'}$ is universally closed. By the descent of universal closeness along Kummer log flat covers from last part, we have that $\Delta_{X/S}$ is universally closed. The finite morphism $X\times_SX\to \mathring{X}\times_{\mathring{S}}\mathring{X}$ is clearly universally closed, hence $\Delta_{\mathring{X}/\mathring{S}}$ as the composition $X\xrightarrow{\Delta_{X/S}} X\times_SX\to \mathring{X}\times_{\mathring{S}}\mathring{X}$ is also universally closed. By \cite[\href{https://stacks.math.columbia.edu/tag/01KJ}{Tag 01KJ}]{stacks-project}, $\Delta_{\mathring{X}/\mathring{S}}$ is an immersion. It follows that $\Delta_{\mathring{X}/\mathring{S}}$ is a closed immersion, i.e. $X$ is separated over $S$. 

(4) By \cite[\href{https://stacks.math.columbia.edu/tag/04XU}{Tag 04XU}]{stacks-project}, the universally closed map $f:X\to S$ is quasi-compact. Thus to show that $f$ is of finite type, it suffices to show that it is locally of finite type. 
By \cite[\href{https://stacks.math.columbia.edu/tag/02KX}{Tag 02KX}]{stacks-project}, being locally of finite type is local on the target for the fppf topology. Therefore we may assume that the Kummer log flat cover $g$ is as in the local situation of \cite[Prop. 1.3]{i-n-t1}, i.e. $g$ admits a chart $(Q\to M_{S'},P\to M_S, P\xrightarrow{h} Q)$ with $P$, $Q$ fs monoids, and $h$ Kummer, and thus $g$ admits a factorization $S'\xrightarrow{g_1} S''\xrightarrow{g_2} S$ of $g$ with $S'':=S\times_{\Spec\Z[P]}\Spec\Z[Q]$ and $g_1$ a classical fppf cover. Consider the following diagram
\[\xymatrix{
X'\ar[r]\ar[d]^{f'} &X\times_SS''\ar[r]\ar[d]^{f''} &X\ar[d]^f \\
S'\ar[r]^{g_1} &S''\ar[r]^{g_2} &S
}\]
with the two small squares cartesian. Since $f'$ is of finite type, so is $f''$ by \cite[\href{https://stacks.math.columbia.edu/tag/02KZ}{Tag 02KZ}]{stacks-project}. By \cite[Lem. 2.1]{i-n-t1}, $g_2$ satisfies the condition on the base change map in \cite[Lem. 2.3]{i-n-t1}. From this, we further get that $f$ is locally of finite type. This finishes the proof.
\end{proof}

\section{Decomposition of Kummer log flat torsors under classical finite flat group schemes over henselian local base}
Let $S$ be a locally noetherian fs log scheme, and let $\varepsilon:(\mr{fs}/S)_{\mr{kfl}}\to (\mr{fs}/S)_{\mr{fl}}$ be the forgetful map between these two sites. Let $G$ be a finite flat group scheme over the underlying scheme of $S$, and we endow it with the induced log structure from $S$. The Leray spectral sequence $H^i_{\mr{fl}}(S,R^j\varepsilon_*G)\Rightarrow H^{i+j}_{\mr{kfl}}(S,G)$ gives rise to an exact sequence
\begin{equation}\label{eqD.1}
0\to H^1_{\mr{fl}}(S,G)\to H^{1}_{\mr{kfl}}(S,G)\xrightarrow{\alpha} H^0_{\mr{fl}}(S,R^1\varepsilon_*G).
\end{equation}

\begin{prop}\label{propD}
Assume that the underlying scheme of $S$ is the spectrum of a henselian local ring, and the log structure of $S$ admits a chart $P\to M_S$ for an fs monoid $P$ such that the induced map $P\to M_{S,\bar{s}}/\mc{O}_{S,\bar{s}}^\times$ is an isomorphism, where $s$ denotes the closed point of $S$. Then we have the following.
\begin{enumerate}[(1)]
\item Let $n$ be a positive integer which kills $G$. Then
\begin{align*}
H^0_{\mr{fl}}(S,R^1\varepsilon_*G)=&H^0_{\mr{fl}}(S,\mc{H}om_S(\Z/n\Z(1),G)\otimes_{\Z}(\Gml/\Gm)_{S_{\mr{fl}}})  \\
=&\mr{Hom}_S(\mathbb Z/n\mathbb Z(1),G)\otimes_{\mathbb Z}P^{\mr{gp}}.
\end{align*}
\item The short exact sequence $0\to\Z/n\Z(1)\to\Gml\xrightarrow{n}\Gml\to0$ on $(\mr{fs}/S)_{\mr{kfl}}$ gives rise to a canonical map $\delta:H^0_{\mr{kfl}}(S,\Gml)\to H^1_{\mr{kfl}}(S,\Z/n\Z(1))$. For any $a\in P^{\mr{gp}}$, let $T_a$ denote the Kummer log flat $\Z/n\Z(1)$-torsor over $S$ given by $\delta(a)$. Then the map $\alpha$ from the exact sequence (\ref{eqD.1}) admits a canonical section given by
$$\beta:\mr{Hom}_S(\mathbb Z/n\mathbb Z(1),G)\otimes_{\mathbb Z}P^{\mr{gp}}\to H^{1}_{\mr{kfl}}(S,G),\quad h\otimes_{\Z}a\mapsto h_*T_a.$$
\item We have a canonical (after having fixed $P$) decomposition
\begin{equation}\label{eqD.2}
H^{1}_{\mr{kfl}}(S,G)\cong H^{1}_{\mr{fl}}(S,G)\oplus \mr{Hom}_S(\mathbb Z/n\mathbb Z(1),G)\otimes_{\mathbb Z}P^{\mr{gp}}.
\end{equation}
\end{enumerate}
\end{prop}
\begin{proof}
(1) Let $\mc{G}:=(\Gml/\Gm)_{S_{\mr{fl}}}$. By Kato's theorem, see \cite[Thm. 4.1]{kat2} or \cite[Thm. 3.12]{niz1}, we have an isomorphism
\begin{equation}\label{eqD.3}
\gamma:\mc{H}om_S(\Z/n\Z(1),G)\otimes_{\Z}\mc{G}\xrightarrow{\cong}R^1\varepsilon_*G,
\end{equation}
which is locally given by the map $\beta$ (see \cite[\S 4.3]{kat2}). Consider the following canonical commutative diagram
\begin{equation}\label{eqD.4}
\xymatrix{
\mr{Hom}_S(\mathbb Z/n\mathbb Z(1),G)\otimes_{\mathbb Z}P^{\mr{gp}}\ar[r]\ar[d]  &H^0_{\mr{fl}}(S,\mc{H}om_S(\Z/n\Z(1),G)\otimes_{\Z}\mc{G})\ar[d] \\
\mr{Hom}_s(\mathbb Z/n\mathbb Z(1),G)\otimes_{\mathbb Z}P^{\mr{gp}}\ar[r] &H^0_{\mr{fl}}(s,\mc{H}om_S(\Z/n\Z(1),G)\otimes_{\Z}\mc{G})
}.
\end{equation}
By \cite[\href{https://stacks.math.columbia.edu/tag/09ZI}{Tag 09ZI}]{stacks-project}, we have 
\begin{align*}
&H^0_{\mr{fl}}(S,\mc{H}om_S(\Z/n\Z(1),G)\otimes_{\Z}\mc{G}) \\
=&H^0_{\mr{\acute{e}t}}(S,\mc{H}om_S(\Z/n\Z(1),G)\otimes_{\Z}(\Gml/\Gm)_{S_{\mr{\acute{e}t}}})  \\
\cong &H^0_{\mr{\acute{e}t}}(s,\mc{H}om_S(\Z/n\Z(1),G)\otimes_{\Z}(\Gml/\Gm)_{S_{\mr{\acute{e}t}}})  \\
=&H^0_{\mr{fl}}(s,\mc{H}om_S(\Z/n\Z(1),G)\otimes_{\Z}\mc{G})
\end{align*}
and 
\begin{align*}
\mr{Hom}_S(\mathbb Z/n\mathbb Z(1),G)\otimes_{\mathbb Z}P^{\mr{gp}}\cong\mr{Hom}_s(\mathbb Z/n\mathbb Z(1),G)\otimes_{\mathbb Z}P^{\mr{gp}}.
\end{align*}
Let $k$ be the residue field of $S$. The property of the given chart forces the induced log structure on $s$ to be the one associated to $P\mapsto k,p\mapsto \begin{cases}0,&\text{ if $p\neq0$;}  \\ 1,&\text{otherwise.}\end{cases}$ Hence the lower horizontal map in (\ref{eqD.4}) is clearly an isomorphism. Then part (1) follows from the commutative diagram (\ref{eqD.4}).

(2) Since the map $\gamma$ from (\ref{eqD.3}) is locally constructed in the same way as $\beta$, the diagram
$$\xymatrix{
&\mr{Hom}_S(\Z/n\Z(1),G)\otimes_{\Z}P^{\mr{gp}}\ar[d]^{H^0_{\mr{fl}}(S,\gamma)}_{\cong}\ar[ld]_{\beta}  \\
H^{1}_{\mr{kfl}}(S,G)\ar[r]^\alpha &H^0_{\mr{fl}}(S,R^1\varepsilon_*G)
}$$ 
is commutative. Therefore $\beta$ is a section to $\alpha$ after the identification $H^0_{\mr{fl}}(S,\gamma)$.

Part (3) follows from part (2) and the exact sequence (\ref{eqD.1}).
\end{proof}

\section*{Acknowledgement}
The first named author has received funding from the German Research Foundation (DFG), projects \textit{Vektorb\"undel und lokale Systeme auf nicht-Archimedischen analytischen R\"aumen} (project number 446443754) and \textit{TRR 326 Geometry and Arithmetic of Uniformized Structures} (project number 444845124).

The second named author thanks Professor Ulrich G\"ortz for very helpful discussions and the Deutsche Arberitsgemeinschaft Kleine AG 2018  ``Deformationen abelscher Variet\"aten'', from which he has benefited a lot. He was partially supported by the \textit{Research Training Group 2553} of the German Research Foundation (DFG).

We wish to thank the anonymous referees for very helpful suggestions for improvements.


\begin{thebibliography}{9}
\bibitem{sga3-2}
M. Artin, M. Demazure, A. Grothendieck, M. Raynaud, and J. P. Serre, Sch\'emas en groupes II: Groupes de type multiplicatif, et structure
  des sch\'emas en groupes g\'en\'eraux, S\'eminaire de G\'eom\'etrie
  Alg\'ebrique du Bois Marie 1962/64 (SGA 3), \textit{Lecture Notes in Mathematics, Springer}, 152 (1970).

\bibitem{sga4}
M. Artin, A. Grothendieck, and J.-L. Verdier, Theorie de Topos et
  Cohomologie Etale des Schemas III, 
  \textit{Lecture Notes in Mathematics, Springer}, 305 (1973).

\bibitem{b-c-c1}
A. Bertapelle, M. Candilera, and V. Cristante, Monodromy of
  logarithmic Barsotti-Tate groups attached to 1-motives, \textit{J. Reine Angew.
  Math.}, 573 (2004), {211}--{234}.

\bibitem{b-m1}
A. Bertapelle and N. Mazzari, On deformations of 1-motives,
\textit{Canad. Math. Bull.}, 62 (2019), no. 1, {11}--{22}.

\bibitem{egaIV-4}
J. Dieudonn\'e and A. Grothendieck, \'{E}l\'{e}ments de g\'{e}om\'{e}trie alg\'{e}brique. IV. \'{E}tude locale des sch\'{e}mas et des morphismes de sch\'{e}mas {IV}, \textit{Inst. Hautes \'{E}tudes Sci. Publ. Math.}, 32 (1967).

\bibitem{g-w1}
U. G\"ortz and T. Wedhorn, Algebraic geometry I: Schemes with examples and exercises, \textit{Advanced
  Lectures in Mathematics, Vieweg + Teubner, Wiesbaden}, (2010).

\bibitem{i-n-t1}
L. Illusie, C. Nakayama, and T. Tsuji, On log flat descent, \textit{Proc.
  Japan Acad. Ser. A Math. Sci.}, 89 (2013), no. 1, {1}--{5}.

\bibitem{k-k-n2}
T. Kajiwara, K. Kato, and C. Nakayama, Logarithmic abelian
  varieties, \textit{Nagoya Math. J.}, 189 (2008), {63}--{138}.

\bibitem{k-k-n4}
T. Kajiwara, K. Kato, and C. Nakayama, Logarithmic abelian
  varieties, Part IV: Proper models, \textit{Nagoya Math. J.}, 219 (2015),
  {9}--{63}.

\bibitem{k-k-n7}
T. Kajiwara, K. Kato, and C. Nakayama, Logarithmic abelian
  varieties, Part VII: Moduli, \textit{Yokohama Math. J.}, 67 (2021), {9}--{48}.

\bibitem{kat1}
K. Kato, Logarithmic structures of Fontaine-Illusie, in Algebraic analysis, geometry, and number theory ({B}altimore, {MD}, 1988), Johns Hopkins Univ. Press, (1989), {191}--{224}.

\bibitem{kat2}
K. Kato, Logarithmic structures of Fontaine-Illusie II---Logarithmic flat topology, \textit{Tokyo J. Math.}, 44 (2021), no. 1, {125}--{155}.

\bibitem{kat4}
K. Kato, Logarithmic Dieudonn\'e theory, \textit{preprint} (1992).

\bibitem{k-s1}
K. Kato and T. Saito, On the conductor formula of Bloch, \textit{Publ.
  Math. Inst. Hautes \'Etudes Sci.},  (2004), no. 100, {5}--{151}.

\bibitem{katz1}
N. Katz, Serre-Tate local moduli, in Algebraic surfaces (Orsay,
  1976--78), \textit{Lecture Notes in Math., Springer}, 868, (1981), {138}--{202}.


\bibitem{mes1}
W. Messing, The crystals associated to Barsotti-Tate groups:
  with applications to abelian schemes, \textit{Lecture Notes in Mathematics, Springer}, 264 (1972).

\bibitem{mil1}
J. Milne, \'Etale cohomology, \textit{Princeton Mathematical
  Series}, 33 (1980).

\bibitem{mil2}
J. Milne, Arithmetic duality theorems, BookSurge, LLC,
  Charleston, SC, second ed., (2006).

\bibitem{nak1}
C. Nakayama, Logarithmic \'etale cohomology, \textit{Math. Ann.}, 308
  (1997), no. 3, {365}--{404}.

\bibitem{niz1}
W. Nizio{\l}, $K$-theory of log-schemes. I, \textit{Doc. Math.}, 13
  (2008), {505}--{551}.

\bibitem{ogu1}
A. Ogus, Lectures on logarithmic algebraic geometry, \textit{Cambridge Studies in Advanced Mathematics, Cambridge University Press}, 178 (2018).

\bibitem{ols1}
M. C. Olsson, Logarithmic geometry and algebraic stacks, \textit{Ann. Sci.
  \'{E}cole Norm. Sup.} (4), 36 (2003), no. 5, {747}--{791}.

\bibitem{ray2}
M. Raynaud, 1-motifs et monodromie g\'eom\'etrique, \textit{Ast\'erisque}, 223 (1994), 295--319.

\bibitem{stacks-project}
Stacks Project Authors, Stacks Project, \textit{\url{http://stacks.math.columbia.edu}}, (2022).

\bibitem{tan1}
A. Tani, Log flat descent of finiteness, \textit{Thesis (in Japanese), Tokyo Institute of Technology}, (2014).

\bibitem{zha3}
H. Zhao, Log abelian varieties over a log point, \textit{Doc. Math.}, 22
  (2017), {505}--{550}.

\bibitem{zha4}
H. Zhao, Extending   finite-subgroup schemes of semistable abelian varieties via log-abelian varieties, \textit{Kyoto J. Math.}, 60 (2020), no. 3, 895--910.

\bibitem{zha5}
H. Zhao, Extending tamely   ramified strict 1-motives into k\'et log 1-motives, \textit{Forum Math. Sigma}, 9 (2021), no. e20, {1}--{34}.

\bibitem{zha8}
H. Zhao, The higher direct images of locally constant group schemes
  from the kummer log flat topology to the classical flat topology, \textit{preprint}, (2021).

\end{thebibliography}
\end{document}